\documentclass[11pt]{article}
 
\usepackage{amsmath,epsfig,amssymb,amsbsy,verbatim,array,color,graphics,graphicx}
\usepackage{amssymb,amsmath,amsthm}
\usepackage{graphicx}
\usepackage{subdepth}
\usepackage[margin=0.5in]{geometry}
\usepackage{tikz}
\usetikzlibrary{shapes.geometric, positioning}
\numberwithin{equation}{section}
\usepackage{hyperref}
 
\newtheorem{thm}{Theorem}[section]
\newtheorem{lemma}[thm]{Lemma}

\newtheorem{cor}[thm]{Corollary}

\newtheorem{clm}[thm]{Claim}

\newtheorem{rmk}[thm]{Remark}
 
\definecolor{darkblue}{rgb}{0,0,0.7}
\definecolor{darkgreen}{rgb}{0,0.3,0}
\definecolor{darkred}{rgb}{0.7,0,0}

\newcommand\eps{\varepsilon}

 \def\ex{\textup{ex}}
 
\begin{document}
 
\title{On degree powers in the degenerate Tur\'an problem\\[2ex]}
\author{Ping Hu$^\dag$ \and Ting Lan$^\dag$ \and Henry Liu\!\!
\thanks{Corresponding author\\
\indent\hspace{0.12cm}$^\dag$School of Mathematics, Sun Yat-sen University, Guangzhou 510275, China. E-mail addresses: {\tt huping9@mail.sysu.edu.cn} (P.~Hu), {\tt lant28@mail2.sysu.edu.cn} (T.~Lan)\\
\indent\hspace{0.12cm}$^\ddag$Department of Mathematical Sciences, Beijing Normal-Hong Kong Baptist University, Zhuhai 519087, China. E-mail address: {\tt henrychliu@bnbu.edu.cn}\\
}\,\:$^\ddag$
\\[2ex]
}
\date{9 August 2026}
\maketitle
\begin{abstract}
Given a graph $G$ with degree sequence $d_{1},\ldots,d_{n}$ and a positive real number $p$, let $e_{p}(G)=\sum_{i=1}^{n} d_{i}^{p}$. For a fixed family of graphs $\mathcal F$, let $\ex_{p}(n, \mathcal F)$ denote the maximum value of $e_{p}(G)$ over all $\mathcal F$-free graphs $G$ on $n$ vertices. In 2000, Caro and Yuster introduced the following Tur\'an-type problem: For a positive integer $p$ and a fixed graph $F$, determine $\ex_{p}(n, F)$, and characterize the extremal graphs $G$ on $n$ vertices that attain $\ex_p(n, F)$. Recently, Gao, Liu, Ma and Pikhurko proved that $\textup{ex}_{p}(n, \mathcal F)=(\tau(\mathcal F)-1+o(1))n^p$ for real $p>\frac{1}{1-\alpha}$, where $\mathcal F$ is a degenerate family of graphs with classical Tur\'an number $\textup{ex}(n, \mathcal F)=O(n^{1+\alpha})$ for some $\alpha\in[0,1)$, and $\tau(\mathcal F)$ is the minimum size of an independent vertex cover over all bipartite graphs $F\in\mathcal F$. Based on their method, we obtain a stability result for $\textup{ex}_{p}(n, \mathcal F)$, and prove that all extremal graphs must contain the complete bipartite graph $K_{\tau(\mathcal F)-1,n-\tau(\mathcal F)+1}$ when $n$ is sufficiently large. Our results can be used to deduce all previously known results about $\textup{ex}_{p}(n, F)$ when $F$ is a bipartite graph and $n$ is sufficiently large. We also obtain several new exact results for $\textup{ex}_{p}(n, F)$, namely, when $F$ is an even cycle, a complete bipartite graph, a discrete hypercube, a caterpillar forest, and a spider forest.\\
 
 
\end{abstract}

\section{Introduction}
 
In this paper, all graphs are finite, simple, and undirected. Let $G$ and $H$ be two graphs. Let $v\in V(G)$ be a vertex, and $A,B\subseteq V(G)$ be disjoint subsets. The \emph{neighbourhood} of $v$ is $N_{G}(v):=\{ u\in V(G):uv\in E(G)\}$, and the \emph{degree} of $v$ is $d_{G}(v):=|N_{G}(v)|$. The \emph{maximum degree} of $G$ is denoted by $\Delta(G)$. Let $d_{G}(v,A):=|N_{G}(v)\cap A|$. For a real number $p>0$, we write $d_{G}^p(v):=(d_{G}(v))^p$ and $d_{G}^p(v,A):=(d_{G}(v,A))^p$ for convenience. Let $G[A]$ denote the subgraph of $G$ induced by $A$. We use $kG$ to denote $k$ vertex-disjoint copies of $G$, and $G\vee H$ to denote the \emph{join} of $G$ and $H$, i.e., the graph obtained by adding all edges between vertex-disjoint copies of $G$ and $H$. A \emph{vertex cover} of $G$ is a subset $S$ of $V(G)$ such that every edge in $G$ contains at least one element of $S$. The minimum size of a vertex cover of $G$ is the \emph{vertex cover number} of $G$, denoted by $\beta(G)$. If $G$ is a bipartite graph, an \emph{independent vertex cover} of $G$ is an independent set that is also a vertex cover of $G$. The minimum size of an independent vertex cover of $G$ is the \emph{independent cover number} of $G$, denoted by $\tau(G)$. Let $I(G)$ denote the family of all minimum independent vertex covers of $G$. Let $P_{t}, C_{t}, K_{t}, E_{t}$ and $M_t$ denote the path, the cycle, the complete graph, the empty graph, and a maximum matching on $t$ vertices, respectively. Let $K_{s,t}$ denote the complete bipartite graph with classes of size $s$ and $t$. In particular, we denote the star $K_{1,t}$ by $S_t$. A \emph{broom} graph is obtained by adding pendant edges to one end-vertex of a path.

Let $\mathcal F$ be a family of graphs. If $\mathcal F$ consists of a single graph $F$, we write $F$ instead of $\{F\}$ for simplicity. We say that a graph $G$ is \emph{$\mathcal F$-free} if $G$ does not contain $F$ as a subgraph for every $F\in\mathcal F$. The \emph{Tur\'an number} $\ex(n, \mathcal F)$ of $\mathcal F$ is the maximum number of edges in an $\mathcal F$-free graph on $n$ vertices. Tur\'an’s theorem \cite{T1941} states that $\ex(n, K_{r+1})=e(T_{r}(n))$ for $n\ge r\ge 2$, where $T_{r}(n)$ is the \emph{Tur\'an graph}, i.e., the complete balanced $r$-partite graph on $n$ vertices where every class has size $\lceil\frac{n}{r}\rceil$ or $\lfloor\frac{n}{r}\rfloor$. Moreover, $T_r(n)$ is the unique extremal graph attaining $\ex(n, K_{r+1})$. For a graph $G$ with degree sequence $d_{1},\ldots,d_{n}$ and a real number $p>0$, let $e_{p}(G):=\sum_{i=1}^{n} d_{i}^{p}$. In 2000, Caro and Yuster \cite{CY2000} introduced the following Turán-type problem: For a fixed graph $F$ and a positive integer $p$, determine the function $\ex_{p}(n, F)$, which is the maximum of $e_{p}(G)$ taken over all $F$-free graphs $G$ on $n$ vertices. Moreover, characterize the \textit{extremal graphs} for $\ex_p(n,F)$, i.e., the $F$-free graphs $G$ on $n$ vertices with $e_{p}(G)=\ex_{p}(n, F)$. We may naturally extend $\ex_{p}(n, F)$ by replacing $F$ with a family of graphs $\mathcal F$, and consider the function $\ex_{p}(n, \mathcal F)$. Clearly, we have $\ex_{1}(n, \mathcal F)=2\,\ex(n, \mathcal F)$. In \cite{CY2000}, Caro and Yuster proved that $\ex_{p}(n, K_{r+1})=e_{p}(T_{r}(n))$ for $p=1, 2, 3$. When $F$ is a non-bipartite graph, Pikhurko \cite{P2004} showed that in order to compute $\ex_{p}(n, F)$ asymptotically in $n$, it is enough to consider complete $(\chi(F)-1)$-partite graphs only, where $\chi(F)$ is the \emph{chromatic number} of $F$. 

A family of graphs $\mathcal F$ is \emph{degenerate} if at least one graph of $\mathcal F$ is bipartite. The aim of this paper is to study the function $\ex_p(n,\mathcal F)$ when $\mathcal F$ is a degenerate family of graphs. Numerous results about $\ex_p(n,F)$ have previously been proved for specific bipartite graphs $F$, including even cycles \cite{CCZ2023,G2025b,N2009}, complete bipartite graphs \cite{CY2000,G2025b}, star forests \cite{CY2000,LLQS2019}, linear forests \cite{CY2000,LLQS2019}, and broom graphs \cite{G2025a,LLQS2019,WY2022}. Recently, Gerbner \cite{G2025b} studied the connection between the function $\ex_p(n,F)$ and counting stars $S_p$ in $F$-free graphs. If $\mathcal F$ is a degenerate family of graphs, we define $\tau(\mathcal F):=\min\{\tau(F):F\in\mathcal F$ is bipartite$\}$. The K\H{o}v\'ari-S\'os-Tur\'an theorem \cite{KST1954} implies $\ex(n,\mathcal F)=O\big(n^{2-\frac{1}{\tau(\mathcal F)}}\big)$, and thus $\ex(n,\mathcal F)=O(n^{1+\alpha})$ for some $\alpha\in[0,1)$. In particular, it is well known that if some graph of $\mathcal F$ is a forest, then $\ex(n,\mathcal F)=O(n)$. We shall study $\ex_p(n,\mathcal F)$ when $p$ is a real number above some threshold depending on $\alpha$, namely, $p>\frac{1}{1-\alpha}$. Observe that the complete bipartite graph $K_{\tau(\mathcal F)-1,n-\tau(\mathcal F)+1}$ is $\mathcal F$-free, and thus $\ex_{p}(n, \mathcal F)\ge e_{p}(K_{\tau(\mathcal F)-1,n-\tau(\mathcal F)+1})=(\tau(\mathcal F)-1+o(1)) n^{p}$. Recently Gao, Liu, Ma and Pikhurko \cite{GLMP2025} obtained the following asymptotic result for $\ex_{p}(n, \mathcal F)$.

\begin{thm}\label{thm:GLMP} {\em(\cite{GLMP2025})}
Let $p>1$ be a real number. Suppose that $\mathcal F$ is a degenerate family of graphs satisfying $\ex(n, \mathcal F)=O(n^{1+\alpha })$ for some constant $\alpha\in[0,1)$. Then there exists a constant $C_{\mathcal F}>0$ such that\\[-1.2ex]
\[
  \ex_{p}(n, \mathcal F) 
  \left\{
  \begin{array}{l@{\quad}l}
   \leq C_{\mathcal F}n^{1+p\alpha}, & \textup{\emph{if $1 < p < \frac{1}{1-\alpha}$;}} \\ 
  = (\tau(\mathcal F)-1+o(1)) n^{p}, & \textup{\emph{if $p > \frac{1}{1-\alpha}$.}}
  \end{array}
  \right.
\]
\end{thm}
Gao et al.~actually proved the version of Theorem \ref{thm:GLMP} when $\mathcal F$ is a degenerate family of uniform hypergraphs. By using ideas from the proof of Theorem~\ref{thm:GLMP}, we obtain a stability result for $\ex_{p}(n, \mathcal F)$ involving vertex degrees, as follows.

\begin{thm}\label{thm:stability}
Let $\mathcal F$ be a degenerate family of graphs satisfying $\ex(n, \mathcal F)=O(n^{1+\alpha })$ for some constant $\alpha\in [0,1)$, and $\tau(\mathcal F)\ge 2$. Let $p>\frac{1}{1-\alpha}$. If $G$ is an $\mathcal F$-free graph on $n$ vertices such that $e_{p}(G)\ge (\tau(\mathcal F)-1-o(1)) n^{p}$, then $G$ has exactly $\tau(\mathcal F)-1$ vertices, each of degree $(1-o(1))n$. All remaining vertices of $G$ have degree $o(n)$.
\end{thm}

Theorem~\ref{thm:stability} states that any $\mathcal F$-free graph $G$ with $e_{p}(G)\ge (\tau(\mathcal F)-1-o(1)) n^{p}$ must contain a specific structure: There are exactly $\tau(\mathcal F)-1$ ``hub" vertices, each of which is adjacent to almost all of the other vertices. Thus, $G$ contains a complete bipartite subgraph $K_{\tau(\mathcal F)-1,(1-o(1))n}$. Building upon this structural information, we next determine the structure of the extremal graphs for $\ex_p(n,\mathcal F)$. Our main result says that all extremal graphs contain the complete bipartite subgraph $K_{\tau(\mathcal F)-1,n-\tau(\mathcal F)+1}$ when $n$ is sufficiently large. 

\begin{thm}\label{thm:main}
Let $\mathcal F$ be a finite, degenerate family of graphs satisfying $\ex(n, \mathcal F)=O(n^{1+\alpha})$ for some constant $\alpha \in[0,1)$, and $\tau(\mathcal F)\ge 2$. Let $p>\frac{1}{1-\alpha}$. Then there exists $n_0=n_0(\mathcal F,p)$ such that the following holds for all $n\ge n_0$. If $G$ is an $\mathcal F$-free graph on $n$ vertices with $e_{p}(G)=\ex_p(n,\mathcal F)$, then $G$ contains $K_{\tau(\mathcal F)-1,n-\tau(\mathcal F)+1}$ as a subgraph.
\end{thm}

By Theorem \ref{thm:main}, any extremal graph $G$ for $\ex_p(n,\mathcal F)$ with $n$ sufficiently large must contain $K_{\tau(\mathcal F)-1,n-\tau(\mathcal F)+1}$ as a subgraph. Thus, $G$ admits a partition $V(G)=U\cup X$, where $U$ is a small set of $\tau(\mathcal F)-1$ ``hub'' vertices, and $X$ is a large set of $n-\tau(\mathcal F)+1$ ``reservoir'' vertices. We would like to  determine the structure of the edges of $G$ within $U$ and $X$. If $K_{\tau(\mathcal F)-1,n-\tau(\mathcal F)+1}$ is a \emph{maximal $\mathcal F$-free graph}, i.e., adding any edge to $K_{\tau(\mathcal F)-1,n-\tau(\mathcal F)+1}$ creates a copy of some $F\in\mathcal F$, then $K_{\tau(\mathcal F)-1,n-\tau(\mathcal F)+1}$ must be the unique extremal graph for $\ex_p(n,\mathcal F)$. We immediately have the following corollary.

\begin{cor}\label{cor:H}
Let $\mathcal F$ be a finite, degenerate family of graphs satisfying $\ex(n, \mathcal F)=O(n^{1+\alpha})$ for some constant $\alpha \in[0,1)$, and $\tau(\mathcal F)\ge 2$. Let $p>\frac{1}{1-\alpha}$. Then there exists $n_0=n_0(\mathcal F,p)$ such that the following holds for all $n\ge n_0$. If $G$ is an $\mathcal F$-free graph on $n$ vertices with $e_{p}(G)=\ex_p(n,\mathcal F)$, and $K_{\tau(\mathcal F)-1,n-\tau(\mathcal F)+1}$ is a maximal $\mathcal F$-free graph, then $G=K_{\tau(\mathcal F)-1,n-\tau(\mathcal F)+1}$.
\end{cor}

One\,\, particularly\,\, important\,\, case\,\, is\,\, when\,\, the\,\, induced\,\, subgraph\,\, $G[X]$\,\, is\,\, a\,\, near regular graph. A graph $L$ on $t$ vertices is \emph{near $\ell$-regular} if $L$ is an $\ell$-regular graph when $\ell t$ is even, otherwise $L$ has $t-1$ vertices of degree $\ell$ and one vertex of degree $\ell-1$.  Let $\mathcal L(n,s, \ell)$ denote the family of graphs that are the join of $K_{s}$ and a near $\ell$-regular graph on $n-s$ vertices. For $\ell\in\{0,1\}$, let $L(n,s,\ell)$ denote the only graph of $\mathcal L(n,s, \ell)$. Clearly, all graphs $G\in \mathcal L(n,s, \ell)$ have the same value of $e_p(G)$. For a bipartite graph $F$, define $\ell(F):=\min_{S\in I(F)}\{\min_{u\in S}d_F(u)\}$. For a degenerate family of graphs $\mathcal F$, define\\[-1.2ex]
\[
\ell(\mathcal F):=\min\{\ell(F): F\in\mathcal F\textup{ is bipartite and }\tau(F)=\tau(\mathcal F)\}.
\]
Using Theorem~\ref{thm:main}, we may obtain the following corollary. 
\begin{cor}\label{cor:L}
Let $\mathcal F$ be a finite, degenerate family of graphs satisfying $\ex(n, \mathcal F)=O(n^{1+\alpha})$ for some constant $\alpha \in[0,1)$, and $\tau(\mathcal F)\ge 2$. Let $p>\frac{1}{1-\alpha}$. Then there exists $n_0=n_0(\mathcal F,p)$ such that the following holds for all $n\ge n_0$. Let $G$ be an $\mathcal F$-free graph on $n$ vertices with $e_p(G) = \ex_p(n,\mathcal F)$. If there exists an $\mathcal F$-free graph in $\mathcal L(n,\tau(\mathcal F)-1, \ell(\mathcal F)-1)$, then $G\in \mathcal L(n,\tau(\mathcal F)-1,\ell(\mathcal F)-1)$.
\end{cor}

Note that if $F\in\mathcal F$ where $F$ may be non-bipartite, and the vertex cover number $\beta(F)<\tau(\mathcal F)$, then the graph $L(n,\tau(\mathcal F)-1,0)$ contains $F$. Thus, if there exists an $\mathcal F$-free graph in $\mathcal L(n,\tau(\mathcal F)-1,\ell(\mathcal F)-1)$, then we have $\beta(F)\ge \tau(\mathcal F)$. In particular, if $\tau(F)=\tau(\mathcal F)$, then $\beta(F)=\tau(F)=\tau(\mathcal F)$.

Using Theorem~\ref{thm:main} and Corollaries \ref{cor:H}, \ref{cor:L}, we may obtain all previously known exact results about $\ex_p(n,F)$ for a bipartite graph $F$ with $\tau(F)\ge 2$, and sufficiently large $n$. We may also obtain many new results. Below we highlight three cases. The following result on complete bipartite graphs $K_{s,t}$ is an improvement on a result of Gerbner \cite{G2025b}.

\begin{thm}\label{thm:Kst}
Let $2\le s\le t$ and $p>s$. Then there exists $n_0=n_0(s,t,p)$ such that the following holds for all $n\ge n_0$. Let $G$ be a $K_{s,t}$-free graph on $n$ vertices. Then $G$ is an extremal graph  with $e_{p}(G)=\ex_p(n,K_{s,t})$ if and only if $G=K_{s-1}\vee L$ for some near $(t-1)$-regular graph $L$ on $n-s+1$ vertices which does not contain any  $K_{a,b}$ such that $2\le a\le \min\{b,s\}$ and $a+b=t+1$.
\end{thm}

Another important case for $G$ in Theorem \ref{thm:main} is when the induced subgraph $G[X]$ consists of one single edge. Let $E_t^+$ denote the graph on $t$ vertices with one edge, and let $E(n,s):=K_{s}\vee E_{n-s}^+$. For the even cycle $C_{2k}$, Nikiforov~\cite{N2009} proved that $\ex_p(n, C_{2k})=(k-1+o(1))n^p$ for $p\geq 2$. More recently, Gerbner \cite{G2025b} proved that $\ex_p(n, C_4)=e_p(L(n,1, 1))$ for $p\geq 3$ and $n$ sufficiently large. Here, we may improve these results as follows.

\begin{thm}\label{thm:C2k}
\indent\\[-2.7ex]
\begin{enumerate}
\item[(a)] Let $p>2$. Then there exists $n_0=n_0(p)$ such that, for all $n\ge n_0$, we have $\ex_p(n, C_4)=e_p(L(n,1, 1))$, and $L(n,1, 1)$ is the unique extremal graph.
\item[(b)] Let $k\ge 3$ and $p>\frac{k}{k-1}$. Then there exists $n_0=n_0(k,p)$ such that, for all $n\ge n_0$, we have $\ex_p(n, C_{2k})=e_p(E(n,k-1))$, and $E(n,k-1)$ is the unique extremal graph.
\end{enumerate}
\end{thm}

The \emph{$d$-dimensional discrete hypercube} $Q_d$ is the graph on $2^d$ vertices where the vertex set consists of all $(0,1)$-vectors of length $d$, and two vertices are adjacent if and only if their vectors differ at exactly one coordinate. Note that $Q_d$ is $d$-regular, $\tau(Q_d)=2^{d-1}$, and $\ell(Q_d)=d$. The determination of the Tur\'an number of $Q_d$ is a problem that was suggested by Erd\H{o}s \cite{E1964}, and has attracted significant attention. Erd\H{o}s and Simonovits \cite{ES1969} proved that $\ex(n,Q_3)=O(n^{8/5})$. For general $d$, results of F\"uredi \cite{F1991}, and Alon, Krivelevich and Sudakov \cite{AKS2003}, imply that $\ex(n,Q_d)=O_d(n^{2-\frac{1}{d}})$. Very recently, Janzer and Sudakov \cite{JS2024} proved that $\ex(n,Q_d)=O_d(n^{1+\alpha_d})$, where $\alpha_d :=1-\frac{1}{d-1}+\frac{1}{(d-1)2^{d-1}}$. Using the results of \cite{ES1969,JS2024}, we may obtain the following result.

\newpage

\begin{thm}\label{thm:hypercube}
\indent\\[-2.7ex]
\begin{enumerate}
\item[(a)] Let $p>\frac{5}{2}$. Then there exists $n_0=n_0(p)$ such that the following holds for all  $n\ge n_0$. Let $G$ be a $Q_3$-free graph on $n$ vertices. Then $G$ is an extremal graph with  $e_p(G)=\ex_p(n,Q_3)$ if and only if $G=K_3\vee L$ for some graph $L$ on $n-3$ vertices which is a disjoint union of copies of $C_3$ and $C_4$.
\item[(b)] Let $d\ge 2$, and $\alpha_d :=1-\frac{1}{d-1}+\frac{1}{(d-1)2^{d-1}}$. Let $p>\frac{1}{1-\alpha_d}$. Then there exists $n_0=n_0(d,p)$ such that the following holds. For all $n\ge n_0$ with $n\equiv (2^{d-1}-1)$ \textup{(mod }$d)$, if $G$ is a $Q_d$-free graph on $n$ vertices with  $e_{p}(G)=\ex_p(n,Q_d)$, then $G\in \mathcal L(n,2^{d-1}-1,d-1)$.
\end{enumerate}
\end{thm}

\begin{rmk}\label{rmk:star}
\textup{It is easy to see that if $F$ is a bipartite graph, then $\tau(F)=1$ if and only if $F$ consists of a star and possibly some isolated vertices. For the case of the star $S_t$, Caro and Yuster \cite{CY2000} showed that $\ex_p(n,S_t)=e_p(L)$ for real $p\ge 1$ and $n\ge t$, where $L$ is a near $(t-1)$-regular graph on $n$ vertices. Moreover, the extremal graphs are precisely all such graphs $L$. Hence, if $\mathcal F$ is a degenerate family of bipartite graphs such that $\tau(F)=1$ for all $F\in\mathcal F$, then $\ex_p(n,\mathcal F)=\ex_p(n,S_t)$ for $n$ sufficiently large, where $S_t$ is a star among $F\in\mathcal F$ with $t\ge 1$ minimum. Moreover, the extremal graphs for $\ex_p(n,\mathcal F)$ are the same as those for $\ex_p(n,S_t)$. However, if $\mathcal F$ is a degenerate family of graphs with $\tau(\mathcal F)=1$, and some $F\in\mathcal F$ is either bipartite and satisfies $\tau(F)\ge 2$, or is non-bipartite, then the determination of $\ex_p(n,\mathcal F)$ becomes another problem.}
\end{rmk}

The rest of this paper is organized as follows. We prove Theorem~\ref{thm:stability} in Section~\ref{sec:stabilityproof}. Based on Theorem~\ref{thm:stability}, we prove in Section~\ref{sec:mainproof} our main result, Theorem \ref{thm:main}, as well as Corollary \ref{cor:L}. In Section~\ref{sec:applicaion}, we prove  Theorems~\ref{thm:Kst}, \ref{thm:C2k} and \ref{thm:hypercube}. We also prove further exact results on $\ex_p(n,F)$ for some other bipartite graphs $F$, namely, disjoint unions of even cycles, caterpillar forests, and spider forests.

\section{Proof of Theorem~\ref{thm:stability}}\label{sec:stabilityproof}

Before we prove Theorem \ref{thm:stability}, we shall recall the classical Zarankiewicz problem~\cite{Z1951}. The problem asks for the determination of the \emph{Zarankiewicz number} $z(m,n;s,t)$, which is the maximum number of edges in a bipartite graph with classes $A$ and $B$,  $|A|=m$, $|B|=n$, and not containing a copy of $K_{s,t}$ with $s$ vertices in $A$ and $t$ vertices in $B$. The following result of K\H{o}v\'ari, S\'os and Tur\'an~\cite{KST1954} on this problem will be used in the proof of Theorem~\ref{thm:stability}.

\begin{thm}[K\H{o}v\'ari, S\'os, Tur\'an~\cite{KST1954}]\label{lem:KST}
Let $m,n,s,t\ge 1$ be integers. Then
\[
  z(m, n;s, t) \leq (t-1)^{\frac{1}{s}} m n^{1-\frac{1}{s}}+(s-1) n.
\]
\end{thm}

We will also require the following simple inequality.

\begin{lemma}\label{lem:fixsum}
 Let $p\ge 1$ and $a\ge b\ge c\ge 0$. Then
\[
  a^{p}+b^{p}\le (a+c)^{p}+(b-c)^{p}.
\]
\end{lemma}
\begin{proof}
Let $f(x)={(a+x)}^{p}+{(b-x)}^{p}$ for $x\in[0,b]$. Then $f(x)$ is a non-decreasing function since ${f}'(x)=p \big({(a+x)}^{p-1}-{(b-x)}^{p-1} \big)\ge 0$. Hence, $f(c)\ge f(0)$ for all $c\in[0,b]$, which gives the required inequality.
\end{proof}

We may now prove Theorem \ref{thm:stability}.

\begin{proof}[Proof of Theorem~\ref{thm:stability}]

We will follow ideas from the proof of Proposition 4.2 in Gao et al.~\cite{GLMP2025}. Let $F
\in\mathcal F$ be a bipartite graph such that $\tau(F)=\tau(\mathcal F)\ge 2$. Let $V(F)=S\cup T$ be a bipartition of $F$ with $|S|=s=\tau(F)=\tau(\mathcal F)$ and $|T|=t$, so that $2\le s\le t$. Let $p_\ast :=\frac{1}{1-\alpha }$. Choose sufficiently small constants $\delta_1,\delta_2>0$ such that 
\[
  \delta_{1} < 
  \begin{cases} 
  \min\left\{ p-p_\ast, \frac{\alpha}{p-p_\ast} \right\}, & \text{if} ~\alpha \in (0,1); \\
  p-p_\ast = p-1, & \text{if} ~\alpha = 0;
  \end{cases}
  \quad \text{and} \quad 
  \delta_{2}:= \delta_{1} + \frac{\delta_{1}}{1 - \alpha} < \min\left\{\frac{1}{s}, \frac{p-1}{p} \right\}.
\]
Let $G$ be an $\mathcal F$-free graph on $n$ vertices as stated in Theorem \ref{thm:stability}. Let 
\[
A:=\big\{u\in V(G):d_{G}(u)\ge n^{1-\delta_{1}}\big\},\quad B :=V(G) \setminus A.
\]
Fix a sufficiently small constant $0< \varepsilon <\min\big\{4^{-\frac{p}{p-1}},\frac{1}{3s+1}\big\}<\frac{1}{4}$, and let $n$ be sufficiently large, depending on $\alpha, p,s,t,\delta_1,\delta_2$ and $\eps$. We prove several claims.
\begin{clm}[\cite{GLMP2025}, Claim 4.3]\label{clm:stabilityU}   
    $|A|\le n^{\delta_{2}}$.
\end{clm}

\begin{proof}
Suppose to the contrary that $|A|> n^{\delta_{2}}$. Fix a set $V_{1}\subseteq A$ of size $n^{\delta_{2}}$. We choose a set $V_{2}\subseteq V(G)$ of size $n^{\delta_2}$, uniformly at random. Let $H:=G[V_1\cup V_2]$ be the induced subgraph of $V_{1}\cup V_{2}$ in $G$. For $v\in V_1$, the expectation $\mathbb E[d_{H}(v)]$ satisfies
\[
  \mathbb E[d_{H}(v)] \geq \mathbb E[|N_G(v)\cap V_2|]=n^{\delta_2}\cdot \frac{d_G(v)}{n}\ge
  n^{\delta_{2}}\cdot\frac{n^{1-\delta_1}}{n} = n^{\delta_{2}-\delta_{1}}.
\]
Then, according to the linearity of expectation, we have
\[
  \mathbb E\bigg[\frac{1}{2}\sum_{v\in V_{1}}d_{H}(v)\bigg] = 
  \frac{1}{2}\sum_{v\in V_{1}}\mathbb E[d_{H}(v)] \geq 
  \frac{1}{2}n^{\delta_{2}} \cdot n^{\delta_{2}-\delta_{1}} = \frac{n^{(1-\alpha)\delta_{1}}}{2^{2+\alpha}} \cdot 
  (2n^{\delta_{2}})^{1+\alpha} > \ex(2n^{\delta_{2}}, \mathcal F).
\]
Therefore, there exists a choice of $V_{2}$ such that $e(H)>\ex(2n^{\delta_{2}}, \mathcal F)\ge\ex(|V_1\cup V_2|, \mathcal F)$. This means that $H$ is not $\mathcal F$-free, a contradiction.
\end{proof}

\begin{clm}\label{clm:stabilityUdegree}
  $\sum_{u\in A} d_{G}(u) \le (s-1+\varepsilon)n.$
\end{clm}
\begin{proof}
We may assume that $A\neq\emptyset$. If $B=\emptyset$, then $d_G(v)\ge n^{1-\delta_1}$ for all $v\in V(G)$, and $e(G)\ge\frac{1}{2}n\cdot n^{1-\delta_1}=\frac{1}{2}n^{2-\delta_1}\gg n^{1+\alpha}$, since $\delta_2=\delta_1+\frac{\delta_1}{1-\alpha}<\frac{1}{2}$ implies $2-\delta_1>1+\alpha$. This contradicts $e(G)\le\ex(n,\mathcal F)=O\big(n^{1+\alpha}\big)$. Thus, we have  $B\neq\emptyset$. Let $q$ be the number of edges of $G$ between $A$ and $B$. Note that $F\subseteq K_{s,t}$ and $G[A\cup B]$ is $F$-free. By Theorem~\ref{lem:KST} and $|A|\le n^{\delta_{2}}$ from Claim~\ref{clm:stabilityU}, we have
\begin{align*}
\sum_{u\in A} d_{G}(u) &=2e(G[A])+q\le n^{2\delta_2}+z(|A|,|B|; s, t)\\
&\leq n^{2\delta_2}+(t-1)^{\frac{1}{s}}|A||B|^{1-\frac{1}{s}} + (s-1)|B| \\
  &\leq n^{2\delta_2}+(t-1)^{\frac{1}{s}} n^{\delta_{2}} n^{1-\frac{1}{s}} + (s-1) n \\
  &\leq (s-1+\varepsilon) n.
\end{align*}
The last inequality is due to $2\delta_{2}<1$ and $\delta_2+1-\frac{1}{s}<1$, since $\delta_{2}<\frac{1}{s}$ and $s\ge 2$. 
\end{proof}

\begin{clm}\label{clm:stabilityW}
$\sum_{u \in B} d_{G}^p(u) \leq \varepsilon n^{p}.$
\end{clm}
\begin{proof}
Let\\[-1.2ex]
\[
  \hat{p}:= \left\{
  \begin{array}{ll}
  \frac{1-(p-p_{*}) \delta_{1}}{1 - \alpha} < p_{*} < p,& \textrm{if } \alpha \in (0,1); \\
  p_{*} = 1, & \textrm{if } \alpha = 0.
  \end{array}
  \right.
\]

Firstly, assume that $\sum_{u \in B} d_G^p(u,B)>\frac{\varepsilon} {2^{p}} n^{p}$. If $0< \alpha <1$, we have $1<\hat{p}<\frac{1}{1-\alpha }$ since $0<\delta_{1}<\frac{\alpha}{p-p_{*}}$. Note that
\[
  \frac{\varepsilon} {2^{p}} n^{p} < \sum_{u \in B} d_G^p(u,B) = \sum_{u \in B} d_G^{\hat{p} + (p - \hat{p})}(u,{B}) \leq \sum_{u \in B} d_G^{\hat{p}}(u,B) \cdot (n^{1 - \delta_{1}})^{p - \hat{p}}.
\]
Thus, we obtain
\[
  \sum_{u \in B} d_G^{\hat{p}}(u,B) > \frac{\frac{\varepsilon}{2^{p}} n^{p}}{(n^{1-\delta_{1}})^{p-\hat{p}}} = \frac{\varepsilon}{2^{p}} \cdot n^{1+\hat{p}\alpha+(p_{*}-\hat{p})\delta_{1}}\gg n^{1+\hat{p}\alpha }.
\]
By Theorem~\ref{thm:GLMP}, we have $\sum_{u\in B}d_G^{\hat{p}}(u)\ge\sum_{u \in B} d_G^{\hat{p}}(u,B)>\ex_{\hat{p}} (n, \mathcal F)$, a contradiction. 
If $\alpha =0$, similarly, we have
\[
  \frac{\varepsilon} {2^{p}} n^{p} < \sum_{u \in B} d_G^p(u,B) = \sum_{u \in B} d_G(u,B)\cdot d_G^{p-1}(u,B) \leq \sum_{u \in B} d_G(u,B) \cdot (n^{1 - \delta_{1}})^{p-1},
\]
so that $\sum_{u \in B} d_G(u,B)>\frac{\varepsilon}{2^{p}} n^{1+(p-1) \delta_{1}}\gg n$. Thus, $\sum_{u \in B} d_G(u,B)>2\,\ex(n, \mathcal F)$, another contradiction. 

Therefore, we have $\sum_{u \in B} d_G^p(u,B)\le \frac{\varepsilon}{2^{p}} n^{p}$. Since $|A|\le n^{\delta_2}$ from Claim \ref{clm:stabilityU}, we have
  \begin{align*}
  \sum_{u \in B} d_{G}^p(u) &= \sum_{u \in B} (d_G(u,A) + d_G(u,B))^{p} \leq 2^{p-1} \left( \sum_{u \in B} d_G^p(u,A) + \sum_{u \in B} d_G^p(u,B) \right) \\
  &\leq 2^{p-1} \left( n^{1+\delta_{2}p} + \frac{\varepsilon}{2^{p}} n^{p} \right)\leq \varepsilon n^{p}.
  \end{align*}
  The first inequality follows from Jensen's inequality on the convex function $x^p$ for $x\ge 0$. The last inequality is due to $\delta_{2}<\frac{p-1}{p}$.
\end{proof}

Now we complete the proof of Theorem \ref{thm:stability}. Let 
\[
U:=\big\{u\in V(G):d_{G}(u)\ge (1-{3\varepsilon})n\big\}\subseteq A.
\]
We will show that $|U|= s-1$. Assume that $|U|\le s-2$. Then, by Lemma~\ref{lem:fixsum} and Claim~\ref{clm:stabilityUdegree}, and recalling that $s\ge 2$ and $0< \varepsilon < 4^{-\frac{p}{p-1}}<\frac{1}{4}$, we have 
\begin{equation*}
      \sum_{u \in A} d_{G}^p(u)\le (s-2) n^{p}+(1-{3\varepsilon})^{p} n^{p}+(4\varepsilon)^{p} n^{p} \le (s-1-2\varepsilon) n^{p}.
\end{equation*}
Together with Claim~\ref{clm:stabilityW}, we have
\[
e_p(G) = \sum_{u \in A} d_{G}^p(u) + \sum_{u \in B} d_{G}^p(u) \le (s-1-\varepsilon) n^{p}.
\]
This is a contradiction to $e_p(G)\ge(s-1-o(1))n^{p}$. Also, if $|U|\ge s$, then since $\eps<\frac{1}{3s+1}$,
\[
\sum_{u\in A}d_G(u) \ge \sum_{u\in U}d_G(u)\ge (1-3\eps)sn>(s-1+\eps)n,
\]
which contradicts Claim~\ref{clm:stabilityUdegree}. Therefore, we have $|U|=s-1$. Finally, by Claim~\ref{clm:stabilityUdegree}, we have $d_G(v)\le (3s-2)\eps n$ for $v\in A\setminus U$. Also, $d_G(v)<n^{1-\delta_1}$ for $v\in B$. Therefore, we have $d_G(u)=(1-o(1))n$ for all $u\in U$, and $d_G(v)=o(n)$ for all $v\not\in U$. This completes the proof of Theorem~\ref{thm:stability}.
\end{proof}

\section{Proof of the main result}\label{sec:mainproof}
In this section we prove Theorem~\ref{thm:main} and Corollary \ref{cor:L}.
\begin{proof}[Proof of Theorem \ref{thm:main}.]
Since $\mathcal F$ is finite, we may choose a sufficiently large $n_0=n_0(\mathcal F,p)$, and let $n\ge n_0$. Since $e_{p}(G) =\ex_p(n,\mathcal F)\ge (\tau(\mathcal F)-1-o(1)) n^{p}$, let $U$ be the set of the $\tau(\mathcal F)-1$ vertices of $G$ given by Theorem \ref{thm:stability}. Then, $\min_{u\in U}d_G(u)=(1-o(1))n$. Let
\[
   X: = \bigcap_{u \in U} N_{G}(u),\quad  Y: = V(G) \setminus (U \cup X).
\]
Let $H$ be the graph obtained from $G$ by removing all edges between $X$ and $Y$, all edges inside $Y$, and then adding all missing edges between $U$ and $Y$. Note that $|V(F)|$ is bounded over $F\in\mathcal F$ since $\mathcal F$ is finite. For any $F\in\mathcal F$, suppose that $H$ contains $F$ as a subgraph. Let $W$ be the vertex set of a copy of $F$ in $H$. If $W\cap Y=\emptyset$, then the same copy of $F$ lies in $G$, a contradiction. Suppose that $W\cap Y\neq\emptyset$. Since $|X|=(1-o(1))n$, $|Y|=o(n)$ and $n\ge n_0$, we may choose a subset $Z\subseteq X\setminus W$ of size $|W\cap Y|$. In $H$, vertices of $W\cap Y$ have neighbours only in $U$, while in $G$, every vertex of $Z$ is adjacent to every vertex of $U$. Therefore, $G[(W\setminus Y) \cup Z]$ also contains a copy of $F$, a contradiction. Thus, $H$ is $\mathcal F$-free. We will show that if $Y$ is not empty, then $e_p(G)<e_p(H)$, which will contradict  $e_p(G)=\ex_p(n,\mathcal F)$. Hence, $Y$ must be empty, so $G$ must contain $K_{\tau(\mathcal F)-1,n-\tau(\mathcal F)+1}$. 

Choose a bipartite graph $F\in\mathcal F$ that satisfies $\tau(F)=\tau(\mathcal F)$ and $\ell(F)=\ell(\mathcal F)$. Let $\ell :=\ell(F)-1=\ell(\mathcal F)-1$, $s:=\tau(F)=\tau(\mathcal F)\ge 2$, and $t:=|Y|\ge 1$. If there exists a vertex $v\in X\cup Y$ such that $d_{G}(v,X)\ge \ell+1$, then there exists a copy of $F$ in $G$. So we have $d_G(v,X) \le \ell$ for any $v\in X\cup Y$, and $\sum_{v \in X} d_G(v,Y)=\sum_{v\in Y}d_G(v,X)\le \ell t$. 
Therefore,
\begin{align}\label{eq:bipartiteC}
    \sum_{v\in X} ( d_G^p(v)-d_H^p(v))&= \sum_{v\in X} \big((s-1+d_G(v,X)+d_G(v,Y))^p-(s-1+d_G(v,X))^p\big) \nonumber \\
    &\le \sum_{v\in X} p(s-1+\ell+t)^{p-1} d_G(v,Y)\nonumber \\
    & \le p(s-1+\ell+t)^{p-1} \ell t=O(t^p),
\end{align}
where the first inequality follows from the convexity of the function $x^p$ on $x\ge 0$.

Next, since $d_G(v,U) \leq s-2$ for any $v \in Y$, counting the number of missing edges between $U$ and $Y$ in $G$ gives $\sum_{v \in U} (t-d_G(v,Y) )=\sum_{v \in Y} (s-1-d_G(v,U) )\ge t$. Thus
\begin{align} 
        \sum_{v\in U} ( d_H^p(v)-d_G^p(v))&=\sum_{v\in U}\big((d_G(v,U\cup X)+t)^p-(d_G(v,U\cup X)+d_G(v,Y))^p\big)\nonumber\\
        &\ge \sum_{v\in U} p(d_G(v,U\cup X)+d_G(v,Y))^{p-1}(t-d_G(v,Y))\nonumber\\
&\ge pt|X|^{p-1}= pt n^{p-1}+o(tn^{p-1}), \label{eq:bipartiteU}
\end{align}
where again the first inequality follows from the convexity of the function $x^p$ on $x\ge 0$.

Now, we obtain upper bounds on $\sum_{v \in Y} d_{G}^p(v)$ and estimate $e_p(G)-e_p(H)$. We consider two cases on $t$.\\[1ex]
\noindent {\textbf{Case 1.}} $1\le t<\frac{1}{2}n^{1-\frac{1}{p}}$.\\[1ex]
\indent For all $v \in Y$, we have $d_G(v,Y) \leq t-1$, $d_G(v,U\cup X) \leq s+\ell-2$, and $d_H(v)=s-1$. Thus, $\sum_{v \in Y} d_{G}^p(v)\le t(t+s+\ell-3)^{p}= t^{p+1}+O(t^{p})$. Together with \eqref{eq:bipartiteC} and \eqref{eq:bipartiteU}, we have 
 \begin{align*}
 e_{p}(G) - e_{p}(H) &= \sum_{v\in U} ( d_G^p(v)-d_H^p(v)) + \sum_{v\in X} ( d_G^p(v)-d_H^p(v)) + \sum_{v\in Y} (d_G^p(v)-d_H^p(v))  \\
 &\leq -pt n^{p-1}-o(tn^{p-1}) + O(t^p)  + t^{p+1}-t(s-1)^p.
 \end{align*}
Since $t=o(n)$, we have $t^p=o(tn^{p-1})$, and since $1\le t<\frac{1}{2}n^{1-\frac{1}{p}}$, we have $t^{p+1}<\frac{1}{2}tn^{p-1}$. Therefore we obtain $e_{p}(G) < e_{p}(H)$ for $n\ge n_0$.\\[1ex]
\noindent {\textbf{Case 2.}} $t\ge \frac{1}{2} n^{1-\frac{1}{p}}$.\\[1ex]
\indent By Theorem \ref{thm:GLMP}, we have $\sum_{v\in Y} d_G^p(v,Y) \le \ex_p(t,\mathcal F)=(s-1+o(1))t^{p} \leq st^{p}$. Again, for all $v\in Y$, we have $d_G(v,U\cup X) \leq s+\ell-2$ and $d_H(v)=s-1$. By Jensen's inequality on the convex function $x^p$ for $x\ge 0$, we have
\begin{align*}
  \sum_{v \in Y} d_{G}^p(v) &= \sum_{v \in Y} (d_G(v,U\cup X) + d_G(v,Y))^{p}\\
  &\leq 2^{p-1} \Bigg( \sum_{v \in Y} d_G^p(v,U\cup X) + \sum_{v \in Y} d_G^p(v,Y) \Bigg)\\
  &\leq 2^{p-1} ( t(s+\ell-2)^{p} + st^{p} )=O(t^p).
\end{align*}
Together with \eqref{eq:bipartiteC} and \eqref{eq:bipartiteU}, we have 
\begin{align*}
 e_{p}(G) - e_{p}(H) &= \sum_{v\in U} ( d_G^p(v)-d_H^p(v)) + \sum_{v\in X} ( d_G^p(v)-d_H^p(v)) + \sum_{v\in Y} ( d_G^p(v)-d_H^p(v))  \\
&\leq -pt n^{p-1}-o(tn^{p-1}) + O(t^p) -  t(s-1)^{p}.
\end{align*}
Again, we have $e_{p}(G) < e_{p}(H)$ for $n\ge n_0$, since $t^p=o(tn^{p-1})$.  \\[1ex]
\indent We conclude that $Y=\emptyset$, and $G$ must contain $K_{\tau(\mathcal F)-1,n-\tau(\mathcal F)+1}$. This completes the proof of Theorem~\ref{thm:main}.
\end{proof}

\begin{proof}[Proof of Corollary \ref{cor:L}]
According to Theorem \ref{thm:main}, we choose $n_0=n_0(\mathcal F,p)$ and let $n\ge n_0$. Then $G$ admits a partition $V(G)=U\cup X$, with $|U|=q$ and $|X|=n'$ where $q:=\tau(\mathcal F)-1$ and $n':=n-\tau(\mathcal F)+1$; and all edges between $U$ and $X$ are in $G$. If there exists a vertex $v\in X$ such that $d_G(v,X)\ge \ell(\mathcal F)$, then $G$ contains some $F\in\mathcal F$ which attains $\ell(\mathcal F)$. Thus, we have $\Delta(G[X])\le r$ where $r:=\ell(\mathcal  F)-1\ge 0$. Let $H:=K_{\tau(\mathcal F)-1}\vee L\in \mathcal L(n,\tau(\mathcal F)-1,\ell(\mathcal F)-1)$ be an $\mathcal F$-free graph. If $rn'$ is even, then $L$ is $r$-regular, and $e_p(G)\le q(n-1)^p+n'(q+r)^p=e_p(H)$. If $rn'$ is odd, then $r\ge 1$, and some vertex of $G[X]$ has degree at most $r-1$. We have $e_p(G)\le q(n-1)^p+(n'-1)(q+r)^p+(q+r-1)^p=e_p(H)$. Since $e_p(G)=\ex_p(n,\mathcal F)\ge e_p(H)$, we have $e_p(G)=e_p(H)$. For this equality to hold, we have $d_G(u)=n-1$ for all $u\in U$; and $G[X]$ is $r$-regular if $rn'$ is even, $G[X]$ has $n'-1$ vertices of degree $r$ and one vertex of degree $r-1$ if $rn'$ is odd. We conclude that $G\in\mathcal L(n,\tau(\mathcal F)-1,\ell(\mathcal F)-1)$. 
\end{proof}

\section{Applications of the main result}\label{sec:applicaion}

In this section, we prove Theorems \ref{thm:Kst},~\ref{thm:C2k}, and \ref{thm:hypercube}. We also determine exact results for $\ex_p(n,F)$, for some other bipartite graphs $F$ with $n$ sufficiently large.

\subsection{Proof of Theorems~\ref{thm:Kst} and \ref{thm:C2k}}
We first prove Theorem~\ref{thm:Kst}.
\begin{proof}[Proof of Theorem \ref{thm:Kst}]
For $K_{s,t}$ with $2\le s\leq t$, we have $\tau(K_{s,t})=s$ and $\ell(K_{s,t})=t$. By the K\H{o}v\'ari-S\'os-Tur\'an theorem \cite{KST1954}, we have $\ex(n,K_{s,t})=O(n^{2-\frac{1}{s}})$. Thus for $p>s$, choose $n_0=n_0(s,t,p)$ according to Corollary \ref{cor:L}, and let $n\ge n_0$. We first show that a graph $H:=K_{s-1}\vee L\in \mathcal L(n,s-1,t-1)$ is $K_{s,t}$-free if and only if $L$ does not contain any $K_{a,b}$ such that $2\le a\le \min\{b,s\}$ and $a+b=t+1$. 

Suppose that $H$ contains a copy of $K_{s,t}$. Then there exist $t+1$ vertices from the $K_{s,t}$, all belonging to $L$. These vertices form a copy of $K_{a,b}$ in $L$, for some $1\le a\le b$ with $a+b=t+1$. Since $L$ does not contain the star $S_t$, we have $a\ge 2$. Moreover, we have $a\le s$ since the $K_{a,b}$ is a subgraph of the $K_{s,t}$. Hence, we have $2\le a\le\min\{b,s\}$.

Conversely, suppose that $L$ contains a copy of $K_{a,b}$, for some $a,b$ such that $2\le a\le \min\{b,s\}$ and $a+b=t+1$. Note that $t=a+b-1>b$ and $(s-a)+(t-b)=s-1$. Thus, we may append $s-a$ and $t-b$ vertices respectively from the copy of $K_{s-1}$ in $H$ to the classes of the $K_{a,b}$, and obtain a copy of $K_{s,t}$ in $H$.

Now, it can be shown that for $n$ sufficiently large, one may construct a near $(t-1)$-regular graph $L$ on $n-s+1$ vertices with arbitrarily large girth $g(L)$ (see for example, \cite{GH2024}, Lemma 2.2). Letting $g(L)\ge 5$ and $H :=K_{s-1}\vee L\in \mathcal L(n,s-1,t-1)$, it follows from above that $H$ is $K_{s,t}$-free. Hence by Corollary \ref{cor:L}, if $G$ is a $K_{s,t}$-free graph on $n\ge n_0$ vertices with $e_p(G)=\ex_p(n,K_{s,t})$, then we must necessarily have $G\in \mathcal L(n,s-1,t-1)$. Moreover, we have $G=K_{s-1}\vee L$ with $L$ not containing any $K_{a,b}$ as described above.
\end{proof}

Secondly, we prove Theorem~\ref{thm:C2k}. We shall in fact prove the following stronger result where $F$ is a disjoint union of even cycles.

\begin{thm}\label{thm:ec}
Let $t\ge 1$, and let $F=\bigcup_{i=1}^t {C_{2k_i}}$ be a disjoint union of even cycles, where $2\le k_1\le k_2\le \cdots \le k_t$, and $b=\sum_{i=1}^t k_{i}$. Let $p>\frac{k_1}{k_{1}-1}$. Then there exists $n_0=n_0(F,p)$ such that the following holds for all $n\ge n_0$. If $G$ is an $F$-free graph on $n$ vertices with $e_p(G)=\ex_p(n, F)$, then
\[
G=
\begin{cases}
L(n,b-1, 1), &\textup{\emph{if $k_{i}=2$ for all $i$;}}\\
E(n,b-1), &\textup{\emph{if $k_t\ge 3$.}}
\end{cases}
\]
\end{thm}

\begin{proof}
Firstly, a famous result of Bondy and Simonovits~\cite{BS1974} states that $\ex(n,C_{2k})=O\big(n^{1+\frac{1}{k}}\big)$. We use this to show that $\ex(n,F)=O\big(n^{1+\frac{1}{k_1}}\big)$. Suppose that $e(H)=\ex(n, F)\gg n^{1+\frac{1}{k_1}}$. Then, there exists a copy of $C_{2k_1}$ in $H$. Deleting all edges incident to the copy of $C_{2k_1}$ from $H$, we have deleted $O(n)$ edges, and thus we obtain a subgraph $H'$ of $H$ with $e(H')\gg n^{1+\frac{1}{k_1}}\ge n^{1+\frac{1}{k_2}}$. Then, there exists a copy of $C_{2k_2}$ in $H'$. Repeating this process $t$ times, since we delete $O(n)$ edges at each step, we can find a copy of $F$ in $H$, a contradiction.

Now, let $p>\frac{k_1}{k_1-1}$. We also have $\tau(F)=b\ge 2$ and $\ell(F)=2$. Choose $n_0=n_0(F,p)$ according to Theorem \ref{thm:main} and Corollary \ref{cor:L}, and let $n\ge n_0$. By Theorem \ref{thm:main}, $G$ contains $K_{b-1,n-b+1}$ as a subgraph. For the case when $k_t\ge 3$, let $U$ and $X$ be the sets of size $b-1$ and $n-b+1$. If $G[X]$ contains two edges, incident or independent, then we have a copy of $F$ in $G$. Indeed, we may let the copy of $C_{2k_t}$ contain two edges of $G[X]$, and any other cycle $C_{2k_i}$ contains $k_i$ vertices in both $U$ and $X$. Thus, $G$ is a subgraph of $E(n,b-1)$. Moreover, $E(n,b-1)$ is $F$-free, since if there is a copy of $F$, then each $C_{2k_i}$ must use at least $k_i$ vertices of $U$, so that $|U|\ge b$, a contradiction. Hence, we have $G=E(n,b-1)$. For the case when $k_i=2$ for all $i$, we may use the same argument to obtain that $L(n,b-1,1)$ is $F$-free. By Corollary \ref{cor:L}, we have $G=L(n,b-1,1)$. 
\end{proof}

\subsection{Discrete hypercubes}

To prove Theorem \ref{thm:hypercube}, we will prove the following result.
\begin{thm}\label{thm:cubehalf}
Let $d\ge 1$, and let $H$ be an induced subgraph of $Q_d$ on $2^{d-1}+1$ vertices. Then $H$ has a connected component on at least $d+1$ vertices.
\end{thm}

Theorem \ref{thm:hypercube} follows from Corollary \ref{cor:L} and Theorem \ref{thm:cubehalf}.

\begin{proof}[Proof of Theorem \ref{thm:hypercube}]
(a) We have $\tau(Q_3)=4$ and $\ell(Q_3)=3$. Recall that $\ex(n,Q_3)=O(n^{8/5})$ from \cite{ES1969}. Thus, for $p>\frac{5}{2}$, choose $n_0=n_0(p)$ according to Corollary \ref{cor:L}, and let $n\ge n_0$. Note that the graphs of $\mathcal L(n,3,2)$ are of the form $K_3\vee L$, where $L$ is a graph on $n-3$ vertices consisting of a disjoint union of cycles. We first show that a graph $H:= K_3\vee L\in \mathcal L(n,3,2)$ is $Q_3$-free if and only if $L$ is a disjoint union of copies of $C_3$ and $C_4$. Let $U$ be the vertex set of the $K_3$.

Suppose first that $H$ contains a subgraph $Q$ which is a copy of $Q_3$. If the connected components of $L$ can only be  $C_3$ or $C_4$, then since $|V(Q)\setminus U|\ge 5$, we see that $V(Q)\cap U$ must be a cut-set of $Q$. Since $Q$ is $3$-connected, every cut-set of $Q$ must have at least $3$ vertices. Moreover, it is easy to show that all cut-sets of $Q$ of size $3$ must be the neighbourhood of a vertex in $Q$. Since $|V(Q)\cap U|\le |U|=3$, we have $U$ itself is a cut-set of $Q$, and  $U=N_Q(v)$ for some $v\in V(Q)$. But then, deleting $U$ from the $Q$ gives a copy of $S_3$ in $L$, which contradicts that $\Delta(L)\le 2$.

Conversely, suppose that $L$ contains a cycle $C_k$ for some $k\ge 5$. Let $U=\{y_1,y_2,z_3\}$, and $y_3y_4z_4z_1z_2$ be a path of length $4$ in $L$. Then $y_1y_2y_3y_4y_1$ and $z_1z_2z_3z_4z_1$ form two disjoint copies of $C_4$ in $H$, and $y_iz_i\in E(H)$ for $1\le i\le 4$. Hence, $H$ contains a copy of $Q_3$.
 
Now, it is clear that for every $n\ge 9$, there exists a graph $L$ on $n-3$ vertices which is a disjoint union of copies of $C_3$ and $C_4$. For such $L$, we have $K_3\vee L$ is $Q_3$-free from above. Hence by Corollary \ref{cor:L}, if $G$ is a $Q_3$-free graph on $n\ge n_0$ vertices with $e_p(G)=\ex_p(n,Q_3)$, then we must necessarily have $G\in \mathcal L(n,3,2)$. Moreover,  we have $G=K_3\vee L$ where $L$ is a disjoint union of copies of $C_3$ and $C_4$.\\[0.5ex]
\indent(b) We have $\tau(Q_d)=2^{d-1}$ and $\ell(Q_d)=d$. Recall that $\ex(n,Q_d)=O_d(n^{1+\alpha_d})$ where $\alpha_d:=1-\frac{1}{d-1}+\frac{1}{(d-1)2^{d-1}}$ from \cite{JS2024}. Thus, for $p>\frac{1}{1-\alpha_d}$, choose $n_0=n_0(d,p)$ according to Corollary \ref{cor:L}. For $n\ge n_0$ and $n\equiv (2^{d-1}-1)$ (mod $d)$, consider the graph $H:=K_{2^{d-1}-1}\vee mK_d\in \mathcal L(n,2^{d-1}-1,d-1)$ for some $m$. If $H$ contains a copy of $Q_d$, then there is a set $S$ of $2^{d-1}+1$ vertices of the $Q_d$ in $mK_d$, and the induced subgraph of the $Q_d$ by $S$ cannot have a connected component with more than $d$ vertices. This is a contradiction to Theorem \ref{thm:cubehalf}. Hence, $H$ is $Q_d$-free. By Corollary \ref{cor:L}, if $G$ is a $Q_d$-free graph on $n$ vertices with $e_p(G) = \ex_p(n, Q_d)$, then $G\in\mathcal L(n, 2^{d-1} - 1,d -1)$.
\end{proof}

We now focus on the proof of Theorem \ref{thm:cubehalf}, for which we shall use an algebraic approach. Let $M$ be a square real matrix. A \emph{principal submatrix} of $M$ is a submatrix of $M$ obtained by deleting the same set of rows and columns. If all eigenvalues of $M$ are real, let $\lambda_1(M)$ denote the largest eigenvalue of $M$. In particular, if $M$ is a symmetric matrix, then all eigenvalues of $M$ are real. We will need the following result.

\begin{thm}[Cauchy's interlace theorem]\label{thm:CIT}
Let $A$ be a symmetric $n\times n$ real matrix, and let $B$ be an $m\times m$ principal submatrix of $A$ for some $m \le n$. If the eigenvalues of $A$ are $\lambda_1\ge\lambda_2\ge\cdots\ge\lambda_n$, and the eigenvalues of $B$ are $\mu_1\ge\mu_2\ge\cdots\ge\mu_m$, then for all $1\le i\le m$,\\[-1.2ex]
\[
\lambda_i\ge\mu_i\ge\lambda_{i+n-m}.
\]
\end{thm}

A short proof of Cauchy's interlace theorem can be found in \cite{F2005}.

Recently, Huang~\cite{H2019} settled the \emph{sensitivity conjecture}. The main result in his paper states that if $H$ is an induced subgraph of $Q_d$ on $2^{d-1} +1$ vertices, then the maximum degree of $H$ satisfies $\Delta(H)\ge\sqrt{d}$. One key ingredient in his proof is the construction of a $(-1,0,1)$-matrix $A_d$ such that by changing every $-1$ entry to $1$, the resulting matrix becomes the adjacency matrix of $Q_d$. He showed that the matrix $A_d$ satisfies $A_d^2=dI$, from which it follows that the eigenvalues of $A_d$ are $\sqrt{d}$ and $-\sqrt{d}$, both with multiplicity $2^{d-1}$. His result then follows from an application of Cauchy's interlace theorem.

We follow a similar approach. Huang defined the matrix $A_d$ iteratively. Here, we use a global definition of $A_d$, which was suggested by  Sawin~\cite{S2019}. For $x=(x_1,\dots,x_d)\in V(Q_d)$ and $1\le i\le d$, let $x^{(i)}$ denote the neighbour of $x$ obtained by flipping the $i^{\textup{th}}$ coordinate of $x$ from 
$x_i$ to $1-x_i$. Define the $2^d\times 2^d$ matrix $A_d$, indexed by $V(Q_d)$, as follows.
\[
(A_d)_{xy}:=
\left\{
\begin{array}{ll}
0, & \textup{if $xy\not\in E(Q_d)$;}\\
(-1)^{x_1+x_2+\cdots+x_{i-1}}, & \textup{if $xy\in E(Q_d)$ and $y=x^{(i)}$.}
\end{array}
\right.
\]
Note that when $i=1$, the empty sum $x_1+x_2+\cdots+x_{i-1}$ is defined to be $0$. Also, $A_d$ is a symmetric matrix, since flipping the $i^{\textup{th}}$ coordinate does not change $x_1+x_2+\cdots+x_{i-1}$. It is clear that by changing every $-1$ entry of $A_d$ to $1$, we obtain the adjacency matrix of $Q_d$. Using this definition of $A_d$, we have the following lemma which implies $A_d^2=dI$. Given a subset $U$ of $V(Q_d)$, let $A_d[U]$ denote the principal submatrix of $A_d$ induced by $U$, i.e., $A_d[U]$ is obtained from $A_d$ by deleting the rows and columns of $A_d$ that correspond to the vertices of $Q_d$ not in $U$.
\begin{lemma}\label{lem:square}
Let $1\le r\le d$. Let $U\subseteq V(Q_d)$ be a vertex subset of size $2^r$ such that $Q_d[U]$ is isomorphic to $Q_r$. Then, we have $A_d[U]^2=rI$. Consequently, all eigenvalues of $A_d[U]$ are equal to either $\sqrt r$ or $-\sqrt r$, and each of these two eigenvalues has multiplicity $2^{r-1}$.
\end{lemma}

\begin{proof}
Fix vertices $x,y\in U$, where $x=y$ is possible. The $(x,y)$-entry of $A_d[U]^2$ is
\begin{equation}
(A_d[U]^2)_{xy}=\sum_{z\in U}A_d[U]_{xz}A_d[U]_{zy}.\label{A2sum}
\end{equation}
The sum (\ref{A2sum}) is over all walks of length $2$ from $x$ to $y$ in $Q_d[U]$.

If there is no walk of length $2$ from $x$ to $y$ in $Q_d[U]$, then \\[-1.2ex]
\[
(A_d[U]^2)_{xy}=0.
\]

If $x=y$, then the walks $x\to z\to x$ of length $2$ from $x$ back to itself in $Q_d[U]$ are obtained by flipping one coordinate and then flipping it back. Since $x$ has $r$ neighbours in $Q_d[U]$, there are $r$ such walks, and each contributes $1$ to the sum (\ref{A2sum}). Hence\\[-1.2ex]
\[
(A_d[U]^2)_{xx}=r.
\]

If $x\neq y$, and there is a walk of length $2$ from $x$ to $y$ in $Q_d[U]$, then $x$ and $y$ differ in exactly two coordinates. Suppose these two coordinates are $i<j$. Then there are exactly two paths of length $2$ from $x$ to $y$:\\[-1.2ex]
\[
x\to x^{(i)}\to y \qquad \textrm{and}
\qquad
x\to x^{(j)}\to y.
\]
Both paths are contained in $Q_d[U]$. The term in the sum (\ref{A2sum}) along the first path is
\[
(-1)^{x_1+\cdots+x_{i-1}}(-1)^{x_1+\cdots+x_{j-1}+1},
\]
where the extra $+1$ appears because the first flip changes coordinate $i$, and $i<j$. The term in the sum (\ref{A2sum}) along the second path is
\[
(-1)^{x_1+\cdots+x_{j-1}}(-1)^{x_1+\cdots+x_{i-1}}.
\]
These two terms are negatives of each other, so together they contribute $0$ to the sum (\ref{A2sum}). Hence,\\[-1.2ex]
\[
(A_d[U]^2)_{xy}=0.
\]
Therefore, we have $A_d[U]^2=rI$. This means that every eigenvalue of $A_d[U]$ is either $\sqrt r$ or $-\sqrt r$. Since the diagonal entries of $A_d[U]$ are all $0$, the trace of $A_d[U]$ is $0$, and hence both $\sqrt r$ and $-\sqrt r$ have multiplicity $2^{r-1}$.
\end{proof}

Next, we have the following lemma about connected induced subgraphs of $Q_d$.
\begin{lemma}\label{lem:subcube}
Let $C\subseteq V(Q_d)$ be a subset of size $c$ such that the induced subgraph $Q_d[C]$ is connected. Then $C$ is contained in an $r$-dimensional subcube of $Q_d$ with
$r\le c-1$.
\end{lemma}

\begin{proof}
Let $T$ be a spanning tree of the connected subgraph $Q_d[C]$. Then $T$ has $c-1$ edges, and each edge of $T$ corresponds to flipping one coordinate of $Q_d$. Let $J$ be the set of such coordinates. Then $|J|\le c-1$. Pick an arbitrary vertex $v\in C$. Then all vertices of $C$ can be reached from $v$ by a path in $T$, and along such paths only coordinates in $J$ may change. Hence, all vertices of $C$ agree outside the coordinates in $J$. Therefore $C$ lies in a hypercube with dimension $|J|\le c-1$.
\end{proof}

We are now ready to prove Theorem~\ref{thm:cubehalf}.
\begin{proof}[Proof of Theorem~\ref{thm:cubehalf}]
Let $W=V(H)$, so that $H=Q_d[W]$. Suppose that every connected component of $H$ has at most $d$ vertices. Let the vertex sets of these components be $C_1,\dots,C_t$. Then the matrix $A_d[W]$ is block diagonal with blocks $A_d[C_1],\dots,A_d[C_t]$. Fix one $C_i$, and write $|C_i|=c\le d$. By Lemma~\ref{lem:subcube}, the set $C_i$ is contained in an $r$-dimensional subcube of $Q_d$ with $r\le c-1\le d-1$. Let $U$ denote the vertex set of this $r$-dimensional subcube. If $r=0$, then $\lambda_1(A_d[U])=0=\sqrt r$. If $r\ge 1$, then by Lemma~\ref{lem:square}, we have $\lambda_{1}(A_d[U])=\sqrt r$. Since $C_i\subseteq U$, we have $A_d[C_i]$ is a principal submatrix of $A_d[U]$. By Cauchy's interlace theorem, we have\\[-1.2ex]
\[
\lambda_{1}(A_d[C_i])\le \lambda_{1}(A_d[U])= \sqrt r\le \sqrt{d-1}<\sqrt d.
\]
Since\, $A_d[W]$\, is\, block\, diagonal,\, its\, largest\, eigenvalue\, is\, the\, maximum\, of\, the\, largest eigenvalues of its blocks. Hence,\\[-1.2ex]
\begin{equation}
\lambda_{1}(A_d[W])=\max_{1\le i\le t} \lambda_{1}(A_d[C_i]) < \sqrt d.\label{Th41eq1}
\end{equation}

On the other hand, let $k=2^d$ and $m=2^{d-1}+1=|W|$. If the eigenvalues of $A_d$ in non-increasing order are $\mu_1\ge\mu_2\ge\cdots\ge\mu_k$, then by Lemma \ref{lem:square}, we have $\mu_{k/2}=\sqrt{d}$. Since $A_d[W]$ is a principal submatrix of $A_d$, by Cauchy's interlace theorem, we have
\begin{equation}
\lambda_{1}(A_d[W])\ge \mu_{1+k-m}=\mu_{k/2}=\sqrt{d}.\label{Th41eq2}
\end{equation}

Inequalities (\ref{Th41eq1}) and (\ref{Th41eq2}) contradict each other. We conclude that $H$ must have a connected component on at least $d+1$ vertices.
\end{proof}

\subsection{Caterpillar forests and spider forests}

Let $s\ge 1$ and $\alpha_1,\alpha_2,\dots,\alpha_s\ge 0$. The \emph{caterpillar} $C(\alpha_1,\alpha_2,\dots,\alpha_s)$ is the graph formed by taking a path $P_s=v_1v_2\cdots v_s$, called the \emph{spine} of $C(\alpha_1,\alpha_2,\dots,\alpha_s)$, and attaching $\alpha_i$ pendant edges to $v_i$, for all $1\le i\le s$. Note that stars ($s=1$), paths ($\alpha_i=0$ for all $1\le i\le s$), and brooms ($\alpha_i=0$ for all $2\le i\le s$) are special cases of a caterpillar. Further special cases are \emph{regular caterpillars} ($\alpha_1=\cdots=\alpha_s=\beta$), \emph{double stars} ($s=2$), and \emph{double brooms} ($\alpha_i=0$ for all $2\le i\le s-1$). Write $C(s;\beta)$ for the regular caterpillar with $\alpha_1=\cdots=\alpha_s=\beta$, and $C(s;\beta,\gamma)$ for the double broom with $\alpha_1=\beta$ and $\alpha_s=\gamma$. We are interested in the determination of $\ex_p(n,F)$ when $F$ is a caterpillar forest. Here, we consider the case when no component of $F$ is a star. We have the following result.

\begin{thm}\label{thm:caterpillar}
Let $t\ge 1$, and $F=\bigcup_{j=1}^t F_j$ be a forest where $F_j$ is a caterpillar for $1\le j\le t$. Let $p>1$.
\begin{enumerate}
\item[(a)] Suppose that every $F_j$ is a regular caterpillar, with $F_j=C(s_j;\beta_j)$ where $s_j\ge 2$ and $\beta_j\ge 1$, for $1\le j\le t$. Then there exists $n_0=n_0(F,p)$ such that, if  $G$ is an $F$-free graph on $n\ge n_0$ vertices with $e_p(G)=\ex_p(n, F)$, then
\begin{equation}\label{eq:caterforest1}
G=
\begin{cases}
L(n,s-1,0), &\textup{\emph{if $\beta_j=1$ for all $j$;}}\\
K_{\tau(F)-1, n-\tau(F)+1}, &\textup{\emph{if $\beta_{j_1}\geq 2$ and $s_{j_1}$ is even for some $j_1$;}}\\
M_{\tau(F)-1}\vee E_{n-\tau(F)+1}, &\textup{\emph{if $\beta_j=1$ when $s_j\neq 3$, and}}\\
&\quad\quad\textup{\emph{$\beta_{j_1}\ge 2$, $s_{j_1}=3$ for some $j_1$;}}\\
E^+_{\tau(F)-1}\vee E_{n-\tau(F)+1}, &\textup{\emph{if $\beta_j=1$ when $s_j$ is even, and}}\\
&\quad\quad\textup{\emph{$\beta_{j_1}\ge 2$, $s_{j_1}\ge 5$ is odd for some $j_1$;}}
\end{cases}
\end{equation}
where $s:=\sum_{j=1}^t s_j$, and $\tau(F)$ is given by $\tau(F)=\sum_{j=1}^t\big(\big\lfloor\frac{s_j}{2}\big\rfloor\beta_j+\big\lceil\frac{s_j}{2}\big\rceil\big)$.
\item[(b)] Suppose that every $F_j$ is a broom, a double star, or a double broom, with $F_j=C(s_j;\beta_j,\gamma_j)$ where $s_j\ge 2$ and $1\leq \beta_j\leq \gamma_j$, for $1\le j\le t$. Then there exists $n_0=n_0(F,p)$ such that, if  $G$ is an $F$-free graph on $n\ge n_0$ vertices with $e_p(G)=\ex_p(n, F)$, then
\begin{equation}\label{eq:caterforest2}
G
\begin{cases}
=K_{\tau(F)-1, n-\tau(F)+1}, &\textup{\emph{if $\beta_{j_1}\ge 2$ and $s_{j_1}$ is even for some $j_1$;}}\\
=L(n,\tau(F)-1,0), &\textup{\emph{if $\beta_j=1$ when $s_j$ is even, and}}\\
&\quad\quad\textup{\emph{$s_{j_1}$ is even for some $j_1$;}}\\
\in \mathcal L(n,2t-1,\beta), &\textup{\emph{if $s_j= 3$ and $1\le \beta_j<\gamma_j$ for all $j$;}}\\
=L(n,\tau(F)-1,1), &\textup{\emph{if $s_j\in\{3,5\}$ for all $j$, $1\le \beta_j<\gamma_j$ when $s_j=3$,}}\\
&\quad\quad\textup{\emph{$\beta_j\ge 2$ when $s_j=5$, and $s_{j_1}=5$ for some $j_1$;}}\\
=E(n,\tau(F)-1), &\textup{\emph{if $s_j$ is odd for all $j$, and $\beta_{j_1}=\gamma_{j_1}=1$, $s_{j_1}=3$,}}\\
&\quad\quad\textup{\emph{or $\beta_{j_1}=1$, $s_{j_1}=5$, or $s_{j_1}\ge 7$, for some $j_1$;}}
\end{cases}
\end{equation}
where $s:=\sum_{j=1}^t s_j$, $\beta:=\min_{1\le j\le t}\beta_j$, and $\tau(F)$ is given by
\begin{equation}
\tau(F)=\sum_{j:\,s_j\textup{ is even}}\Big(\frac{s_j}{2}+\beta_j\Big)+\sum_{j:\,s_j\textup{ is odd}}\Big\lceil\frac{s_j}{2}\Big\rceil.\label{eq:caterbtauF}
\end{equation}
Moreover, for the case $s_j= 3$ and $1\le \beta_j<\gamma_j$ for all $j$, every graph of $\mathcal L(n,2t-1,\beta)$ is an extremal graph for $\ex_p(n,F)$.
\end{enumerate}
\end{thm}

\begin{proof}
Throughout\, this\, proof,\, we\, will\, consider\, the\, graph\, $K_{\tau(F)-1,n-\tau(F)+1}$,\, graphs\, of $\mathcal L(n,\tau(F)-1,\ell(F)-1)$, and the graph $E(n,\tau(F)-1)$. Each of these graphs admits a partition of its vertex set into a small set $U$ of size $\tau(F)-1$ consisting of the ``hub'' vertices, and a large set $X$ of size $n-\tau(F)+1$ consisting of the ``reservoir'' vertices. Recall that $\ex(n,F)=O(n)$. Let $p>1$, and choose $n_0=n_0(F,p)$ according to Theorem \ref{thm:main} and Corollaries \ref{cor:H}, \ref{cor:L}. Let $n\ge n_0$.\\[1ex]
\indent(a) Let us refer the four cases of (\ref{eq:caterforest1}) as (\ref{eq:caterforest1}a) to (\ref{eq:caterforest1}d). It is easy to obtain $\tau(F)=\sum_{j=1}^t\big(\big\lfloor\frac{s_j}{2}\big\rfloor \beta_j+\big\lceil\frac{s_j}{2}\big\rceil\big)$. 

Firstly, suppose (\ref{eq:caterforest1}a) holds. Note that $\tau(F)=s$. If $L(n,s-1,0)$ contains a copy of $F$, then we have a set $S$ of $s+1$ vertices of the copy of $F$ in $X$, so $S$ is an independent set in $L(n,s-1,0)$. But two vertices of $S$ must be a spine vertex and its leaf neighbour in $F$, which is impossible. Hence, $L(n,s-1,0)$ is $F$-free. By Corollary \ref{cor:L}, we have $G=L(n,s-1,0)$.

Now, consider the cases (\ref{eq:caterforest1}b), (\ref{eq:caterforest1}c) and (\ref{eq:caterforest1}d), so that $\beta_{j_1}\ge 2$ for some $j_1$. By Theorem \ref{thm:main}, $G$ contains $K:=K_{\tau(F)-1,n-\tau(F)+1}$ as a subgraph, and $K$ is $F$-free. It is easy to obtain $\ell(F)=1$, and since $G$ is $F$-free, we see that $G$ cannot contain an edge within $X$.

Suppose first that (\ref{eq:caterforest1}b) holds. If there is an edge $e$ in $U$, then we have a copy of $F$ with exactly $\sum_{j=1}^t\big\lfloor\frac{s_j}{2}\big\rfloor-1$ spine vertices in $X$, and $e$ is an edge of the spine on $s_{j_1}$ vertices. The number of leaves adjacent to the spine vertices in $X$, along with the spine vertices in $U$, is $\big(\sum_{j=1}^t\big\lfloor\frac{s_j}{2}\big\rfloor \beta_j-\beta_{j_1}\big)+\big(\sum_{j=1}^t\big\lceil\frac{s_j}{2}\big\rceil+1\big)=\tau(F)-\beta_{j_1}+1\le |U|$. Thus, these vertices can all be in $U$. The remaining leaves of $F$ can all be in $X$.  Hence, $K$ is maximal $F$-free. By Corollary \ref{cor:H}, we have $G=K_{\tau(F)-1,n-\tau(F)+1}$.

Now, suppose (\ref{eq:caterforest1}c) or (\ref{eq:caterforest1}d) holds. If we add into $U$ a maximum matching if (\ref{eq:caterforest1}c) holds, and an edge if (\ref{eq:caterforest1}d) holds, then the new graph will still be $F$-free. Indeed, suppose that there is a copy of $F$. Let $a_j$ be the number of spine vertices of $F_j$ in $X$. If (\ref{eq:caterforest1}c) holds, then $a_j\ge 1$ if $s_j=3$. The number of leaves adjacent to the spine vertices in $X$, along with the spine vertices in $U$, is \\[-1.2ex]
\[
\sum_{j=1}^t (a_j\beta_j+s_j-a_j) =\sum_{j=1}^t(a_j(\beta_j-1)+s_j)\ge\sum_{j:\,s_j\neq 3}s_j+\sum_{j:\,s_j=3}(\beta_j+2)=\tau(F)>|U|,
\]
so not all of these vertices can be in $U$. If (\ref{eq:caterforest1}d) holds, then $a_j\ge\big\lfloor\frac{s_j}{2}\big\rfloor$ if $s_j$ is odd. The number of leaves adjacent to the spine vertices in $X$, along with the spine vertices in $U$, is
\begin{align*}
\sum_{j=1}^t (a_j\beta_j+s_j-a_j) &=\sum_{j=1}^t(a_j(\beta_j-1)+s_j)\\
&\ge\sum_{j:\,s_j\textup{ is even}}s_j+\sum_{j:\,s_j\textup{ is odd}}\Big(\Big\lfloor\frac{s_j}{2}\Big\rfloor(\beta_j-1)+s_j\Big)\\
&=\sum_{j:\,s_j\textup{ is even}}s_j+\sum_{j:\,s_j\textup{ is odd}}\Big(\Big\lfloor\frac{s_j}{2}\Big\rfloor\beta_j+\Big\lceil\frac{s_j}{2}\Big\rceil\Big)=\tau(F)>|U|,
\end{align*}
so again not all of these vertices can be in $U$. On the other hand, if we add into $U$ two incident edges if (\ref{eq:caterforest1}c) holds; and two edges, incident or independent, if (\ref{eq:caterforest1}d) holds, then the new graph will contain a copy of $F$. Indeed, for (\ref{eq:caterforest1}c), the spine of $F_{j_1}$ consists of the two incident edges in $U$. For (\ref{eq:caterforest1}d), the spine of $F_{j_1}$ uses the two edges in $U$ and has $\big\lfloor\frac{s_{j_1}}{2}\big\rfloor-1$ vertices in $X$. For both cases, any other $F_j$ has $\big\lfloor\frac{s_j}{2}\big\rfloor$ spine vertices in $X$. The number of leaves adjacent to the spine vertices in $X$, along with the spine vertices in $U$, is 
\[
\Big(\Big\lfloor\frac{s_{j_1}}{2}\Big\rfloor-1\Big)\beta_{j_1}+\Big\lceil\frac{s_{j_1}}{2}\Big\rceil+1+\sum_{j\neq j_1}\Big(\Big\lfloor\frac{s_j}{2}\Big\rfloor\beta_j+\Big\lceil\frac{s_j}{2}\Big\rceil\Big)=\tau(F)-\beta_{j_1}+1\le |U|,
\]
so all of these vertices can be in $U$. The remaining leaves of $F$ are in $X$. It follows that if (\ref{eq:caterforest1}c) holds, then $G=M_{\tau(F)-1}\vee E_{n-\tau(F)+1}$; and if (\ref{eq:caterforest1}d) holds, then $G=E^+_{\tau(F)-1}\vee E_{n-\tau(F)+1}$.\\[1ex]
\indent(b) Let us refer the five cases of (\ref{eq:caterforest2}) as (\ref{eq:caterforest2}a) to (\ref{eq:caterforest2}e). It is easy to obtain the expression (\ref{eq:caterbtauF}) for $\tau(F)$. Let $y_j,z_j$ be the end-vertices of the spine of $F_j$ incident with $\beta_j,\gamma_j$ pendant edges respectively, for $1\le j\le t$.

Firstly, suppose (\ref{eq:caterforest2}a) holds. Consider $K:=K_{\tau(F)-1,n-\tau(F)+1}$, which is $F$-free. By the existence of $j_1$, it is easy to obtain $\ell(F)=1$, and since $G$ is $F$-free, no edge of $G$ can be in $X$. Now, if there is an edge $e$ of $G$ in $U$, then $G$ has a copy of $F$. Indeed, let the spine of $F_{j_1}$ contain $e$, and $y_{j_1},z_{j_1}\in U$, so that $\frac{s_{j_1}}{2}+1$ spine vertices of $F_{j_1}$ are in $U$. For every other $F_j$ where $j\neq j_1$, the spine of $F_j$ has $\big\lceil\frac{s_j}{2}\big\rceil$ vertices in $U$, with $y_j\in X$ and $z_j\in U$ if $s_j$ is even; and $y_j,z_j\in U$ if $s_j$ is odd. Then, the number of leaves adjacent to the spine vertices in $X$, along with the spine vertices in $U$, is
\[
\frac{s_{j_1}}{2}+1+\sum_{j:\,s_j\textup{ is even, }j\neq j_1}\Big(\frac{s_j}{2}+\beta_j\Big)+\sum_{j:\,s_j\textup{ is odd}}\Big\lceil\frac{s_j}{2}\Big\rceil\le \tau(F)-1=|U|,
\]
so all of these vertices can be in $U$. All remaining vertices of $F$ are in $X$. Hence, $K$ is maximal $F$-free, and by Corollary \ref{cor:H}, we have $G=K_{\tau(F)-1,n-\tau(F)+1}$.

Now, suppose that (\ref{eq:caterforest2}b), (\ref{eq:caterforest2}c) or (\ref{eq:caterforest2}d) holds. We have
\[
\ell(F)=\left\{
\begin{array}{ll}
1, & \textup{if (\ref{eq:caterforest2}b) holds, by the existence of $j_1$;}\\
\beta+1, & \textup{if (\ref{eq:caterforest2}c) holds;}\\
2, & \textup{if (\ref{eq:caterforest2}d) holds, by the existence of $j_1$.}
\end{array}
\right.
\]
In each case, we show that the graphs of $\mathcal L(n,\tau(F)-1,\ell(F)-1)$ are all $F$-free. Firstly, suppose (\ref{eq:caterforest2}b) holds. If there is a copy of $F$ in $L(n,\tau(F)-1,0)$, then we have a set $S$ of vertices of $F$ in $X$, where
\begin{align*}
|S| &\ge |V(F)|-(\tau(F)-1)\\
&=\sum_{j:\,s_j\textup{ is even}}(s_j+\gamma_j+1)+\sum_{j:\,s_j\textup{ is odd}}(s_j+\beta_j+\gamma_j)-\tau(F)+1\\
&=\sum_{j:\,s_j\textup{ is even}}\Big(\frac{s_j}{2}+\gamma_j\Big)+\sum_{j:\,s_j\textup{ is odd}}\Big(\Big\lfloor\frac{s_j}{2}\Big\rfloor+\beta_j+\gamma_j\Big)+1.
\end{align*}
For $j$ such that $s_j$ is odd, let $B_j$ consist of $y_j$ and its leaf neighbours, and define $\Gamma_j$ similarly with $z_j$ and for every $j$. Let $S_j:=V(F_j)\setminus \Gamma_j$ for $j$ such that $s_j$ is even, and $S_j:=V(F_j)\setminus(B_j\cup\Gamma_j)$ for $j$ such that $s_j$ is odd. Then $S$ must contain either $B_j$ for some $j$ with $s_j$ odd, or $\Gamma_j$ for some $j$, or $\big\lfloor\frac{s_j}{2}\big\rfloor+1$ vertices of $S_j$ for some $j$. Any of these contradicts the fact that $S$ is an independent set in $L(n,\tau(F)-1,0)$. Next, suppose (\ref{eq:caterforest2}c) holds. We have $\tau(F)=2t$. Let $H\in \mathcal L(n,2t-1,\beta)$. Suppose that there is a copy of $F$ in $H$. Since $|U|=2t-1$, we may assume that $F_1$ has at most one vertex in $U$. If $z_1\in X$, then $d_F(z_1)\le \Delta(H[X])+1\le \beta+1\le\beta_1+1<\gamma_1+1$, a contradiction. Thus, $z_1$ is the only vertex of $F_1$ in $U$, and so $y_1$ and all $\beta_1+1$ of its neighbours in $F$ must lie in $X$. But then, $d_{H[X]}(y_1)\ge\beta_1+1\ge \beta+1$, which contradicts $\Delta(H[X])\le \beta$. Finally, suppose (\ref{eq:caterforest2}d) holds. Let $0\le r\le t-1$ be the number of $F_j$ with $s_j=3$, so that $\tau(F)=2r+3(t-r)$. Suppose that there is a copy of $F$ in $L(n,\tau(F)-1,1)$. From the conditions on the $F_j$, it is not hard to see that in $U$, there must be at least two vertices of $F_j$ if $s_j=3$, and at least three vertices if $s_j=5$. Then the total number of vertices in $U$ is at least $2r+3(t-r)>\tau(F)-1=|U|$, a contradiction. Hence, by Corollary \ref{cor:L}, we have $G=L(n,\tau(F)-1,0)$ if (\ref{eq:caterforest2}b) holds; $G\in \mathcal L(n,2t-1,\beta)$ if (\ref{eq:caterforest2}c) holds; and $G=L(n,\tau(F)-1,1)$ if (\ref{eq:caterforest2}d) holds. Moreover, for (\ref{eq:caterforest2}c), since $H\in \mathcal L(n,2t-1,\beta)$ above was arbitrary, all graphs of $\mathcal L(n,2t-1,\beta)$ are extremal graphs for $\ex_p(n,F)$.

Finally,\,\, suppose\,\, (\ref{eq:caterforest2}e)\,\, holds.\,\, By\,\, Theorem\,\, \ref{thm:main},\,\, $G$\,\, contains\,\, a\,\, copy\,\, of\,\, $K:=K_{\tau(F)-1,n-\tau(F)+1}$. By the existence of $j_1$, we have $\ell(F)=2$, and since $G$ is $F$-free, we see that $G$ cannot contain two incident edges in $X$. Now, if $G$ has two independent edges $e$ and $f$ in $X$, then there is a copy of $F$. Indeed, for $F_{j_1}$, we have $\big\lceil\frac{s_{j_1}}{2}\big\rceil-1$ spine vertices in $U$, with $e$ being the pendant edge at $y_{j_1}$ if $s_{j_1}\in\{3,5\}$, $z_{j_1}\in U$ if $s_{j_1}=5$, and $y_{j_1},z_{j_1}\in U$ if $s_{j_1}\ge 7$. See Figure 1. Every other $F_j$ has $\big\lceil\frac{s_j}{2}\big\rceil$ spine vertices in $U$. All leaves of $F$ are in $X$. Then, the number of vertices of $F$ in $U$ is exactly $\sum_{j=1}^t \big\lceil\frac{s_j}{2}\big\rceil-1=\tau(F)-1=|U|$. Thus, $G$ is a subgraph of $E(n,\tau(F)-1)$. On the other hand, if $E(n,\tau(F)-1)$ contains a copy of $F$, then there would be at least $\sum_{j=1}^t\big\lceil\frac{s_j}{2}\big\rceil>|U|$ vertices of the copy of $F$ in $U$, a contradiction. It follows that $G=E(n,\tau(F)-1)$.\\[0.5ex]
\indent This completes the proof of Theorem~\ref{thm:caterpillar}.
\end{proof}
\indent\\[-0.9cm]
\begin{center}
\begin{tikzpicture}[outer sep=0pt,
    dot/.style={circle, fill, inner sep=1.2pt},
    set/.style={ellipse, draw, minimum height=0.8cm},
    scale=0.67
]

\begin{scope}[xshift=1cm]
    \draw (0,2.4) ellipse [x radius=1.2, y radius=0.6];
    \draw (0,0) ellipse [x radius=2.4, y radius=0.6];
    \node [font=\scriptsize] at (-1.5, 2.4) {$U$};
    \node [font=\scriptsize] at (-2.7, 0) {$X$};
    
    \node[dot] (u1) at (-0.4, 2.4) {};
    \node[dot] (x1) at (-1.6, 0) {};
    \node[dot] (x2) at (-0.8, 0) {};
    \node[dot] (x3) at (0, 0) {};
    \node[dot] (x4) at (0.8, 0) {};
    
    \draw (x1) -- (x2) node[font=\scriptsize, midway, below=-0.08cm] {$e$};
    \draw (x3) -- (x4) node[font=\scriptsize, midway, below=-0.08cm] {$f$};
    \draw (u1) -- (x2) node[font=\scriptsize, midway, left=-0.05cm] {$F_{j_1}$};
    \draw (u1) -- (x3);
    
    \node [font=\scriptsize] at (-0.7,-0.33) {$y_{j_1}$};
    \node [font=\scriptsize] at (-0.02,-0.33) {$z_{j_1}$};
    
    \node at (-0.22, -1.4) {$s_{j_1} = 3$};
\end{scope}

\begin{scope}[xshift=7.5cm]
    \draw (0,2.4) ellipse [x radius=1.6, y radius=0.6];
    \draw (0,0) ellipse [x radius=2.8, y radius=0.6];
    \node [font=\scriptsize] at (-1.9, 2.4) {$U$};
    \node [font=\scriptsize] at (-3.1, 0) {$X$};
    
    \node[dot] (u1) at (-0.8, 2.4) {};
    \node[dot] (u2) at (0.8, 2.4) {};
    \node[dot] (x1) at (-2, 0) {};
    \node[dot] (x2) at (-1.2, 0) {};
    \node[dot] (x3) at (-0.4, 0) {};
    \node[dot] (x4) at (0.4, 0) {};
    
    \draw (x1) -- (x2) node[font=\scriptsize, midway, below=-0.08cm] {$e$};
    \draw (x3) -- (x4) node[font=\scriptsize, midway, below=-0.08cm] {$f$};
    \draw (u1) -- (x2) node[font=\scriptsize, midway, left=-0.05cm] {$F_{j_1}$};
    \draw (u1) -- (x3);
    \draw (u2) -- (x4);
    
    \foreach \a in {-60, -70, -80}
        \draw (u2) -- +(\a:1);
        
    \node [font=\scriptsize] at (-0.9,-0.33) {$y_{j_1}$};
    \node [font=\scriptsize] at (0.45,2.6) {$z_{j_1}$};
    
    \node at (-0.22, -1.4) {$s_{j_1} = 5$};
\end{scope}

\begin{scope}[xshift=15.6cm]
    \draw (0,2.4) ellipse [x radius=2.8, y radius=0.6];
    \draw (0,0) ellipse [x radius=4, y radius=0.6];
    \node [font=\scriptsize] at (-3.1, 2.4) {$U$};
    \node [font=\scriptsize] at (-4.3, 0) {$X$};
    
    \node[dot] (u1) at (-2, 2.4) {};
    \node[dot] (u2) at (-1.2, 2.4) {};
    \node[dot] (u3) at (-0.4, 2.4) {};
    \node[dot] (u4) at (0.4, 2.4) {};
    \node at (1.25, 2.4) {$\dots$};
    \node[dot] (u5) at (2, 2.4) {};
    
    \node[dot] (x1) at (-2.4, 0) {};
    \node[dot] (x2) at (-1.6, 0) {};
    \node[dot] (x3) at (-0.8, 0) {};
    \node[dot] (x4) at (0, 0) {};
    \node[dot] (x5) at (0.8, 0) {};
    \node at (1.65, 0) {$\dots$};
    \node[dot] (x6) at (2.4, 0) {};
    
    \draw (x1) -- (x2) node[font=\scriptsize, midway, below=-0.08cm] {$e$};
    \draw (x3) -- (x4) node[font=\scriptsize, midway, below=-0.08cm] {$f$};
    \draw (u1) -- (x1) node[font=\scriptsize, midway, left=-0.05cm] {$F_{j_1}$};
    \draw (u2) -- (x2);
    \draw (u2) -- (x3);
    \draw (u3) -- (x4);
    \draw (u3) -- (x5);
    \draw (u4) -- (x5);
    \draw (u5) -- (x6);
    
    \foreach \a in {-116.6, -126.6, -136.6}
        \draw (u1) -- +(\a:1);
    \foreach \a in {-63.4}
        \draw (u4) -- +(\a:1);
    \foreach \a in {116.6}
        \draw (x6) -- +(\a:1);
    \foreach \a in {-63.4, -53.4, -43.4}
        \draw (u5) -- +(\a:1);
        
    \node [font=\scriptsize] at (-1.7,2.6) {$y_{j_1}$};
    \node [font=\scriptsize] at (1.65,2.6) {$z_{j_1}$};
        
    \node at (-0.22, -1.4) {$s_{j_1} \ge 7$};
\end{scope}

\end{tikzpicture}
\indent\\[0.2cm]
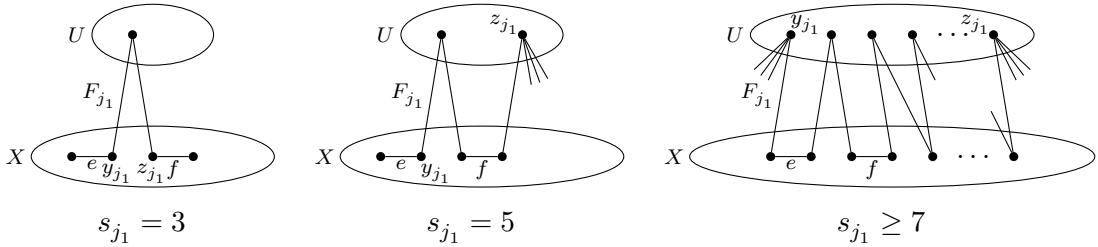
Figure 1. Construction of $F_{j_1}$ in $G$
\end{center}
\indent\\[-1cm]
\begin{rmk}\label{rmk:l=3}
\indent\\[-0.4cm]
\begin{enumerate}
\item[\textup{(a)}] \textup{In Theorem~\ref{thm:caterpillar}(b), we have not completely determined $\ex_p(n, F)$ for the case when $F=\bigcup_{j=1}^tF_j$ is such that $s_j\in\{3,5\}$ for all $1\le j\le t$, with $s_{j_1}=3$ and $2\le \beta_{j_1}=\gamma_{j_1}$ for some $j_1$. For such a forest $F$, it is conceivable that any extremal graph $G$ for $\ex_p(n,F)$ should be of the form $G=K_{\tau(F)-1}\vee L$, where $L$ is a graph on $n-\tau(F)+1$ vertices such that $G$ is $F$-free. However, determining the structure of $L$ appears to be a challenging problem. In particular, if $F=C(3;\beta,\beta)$ where $\beta\ge 2$, then we know that $G=K_1\vee L$, where $L$ must be $2S_{\beta}$-free and of maximum degree at most $\beta$.}
\item[\textup{(b)}] \textup{The case when $F$ is a broom is covered in Theorem~\ref{thm:caterpillar}(b) by setting $t=1$ and $\beta_1=1$. The determination of $\ex_p(n,F)$ for this case was an open problem of Lan et al.~\cite{LLQS2019}. The problem was subsequently solved by Gerbner \cite{G2025a}. Here, by setting $\beta_j=1$ for all $1\le j\le t$ in Theorem~\ref{thm:caterpillar}(b), we have extended Gerbner's result to the answer for $\ex_p(n,F)$ when $F$ is a broom forest.}
\end{enumerate}
\end{rmk}

\indent Let $\ell_1,\ell_2,\dots,\ell_k\ge 1$. The \emph{spider} $S(\ell_1,\ell_2,\dots,\ell_k)$ is the graph formed by attaching $k$ paths of lengths $\ell_1,\ell_2,\dots,\ell_k$, called the \emph{legs} of $S(\ell_1,\ell_2,\dots,\ell_k)$, to a \emph{centre vertex} $v$. Note that stars ($\ell_i=1$ for all $1\le i\le k$; or $k=1$ and $\ell_1\in\{1,2\}$) and paths ($k\in\{1,2\}$) are two special cases of a spider. To ensure that every spider is identified by a unique notation $S(\ell_1,\ell_2,\dots,\ell_k)$, up to permutations of $\ell_1,\ell_2,\dots,\ell_k$, we shall identify the path $P_{\ell+1}$ with $S(\ell)$ if $\ell\neq 2$, and the path $P_3$ with $S(1,1)$. Under these assumptions, we say that $S(\ell_1,\ell_2,\dots,\ell_k)$ is \emph{even-legged} if every $\ell_i$ is even; \emph{odd-legged} if every $\ell_i$ is odd; and \emph{mixed parity} otherwise. Then, every spider is classified as exactly one of these three types. We are interested in the determination of $\ex_p(n,F)$ when $F$ is a spider forest. The case when $F$ is a single star has already been discussed in Remark \ref{rmk:star}. Lan et al.~\cite{LLQS2019} determined $\ex_p(n,F)$ when $F$ is a star forest and a linear forest, where $p\ge 2$ is an integer and $n$ sufficiently large. Here, we have the following result which includes the results of Lan et al.

\begin{thm}\label{thm:Spider}
Let $t\ge 1$, and $F=\bigcup_{j=1}^t F_j$ be a spider forest which is not a star. For $1\le j\le t$, let $F_j=S(\ell_{1,j},\ell_{2,j},\dots,\ell_{k_j,j})$ with centre $v_j$, and $\ell_{1,j},\ell_{2,j},\dots,\ell_{k_j,j}\ge 1$. Let $r_j$ be the number of legs of $F_j$ with length $1$. Assume that $r_j\le 1$ for $1\le j\le s$ and $r_j\ge 2$ for $s<j\le t$, for some $0\le s\le t$. Let \\[-0.7ex]
\[
k:=\sum_{j=1}^t k_j, \quad m:=\min_{1\le j\le s}k_j\textup{\emph{ if }}s\ge 1,\quad r := \sum_{j=1}^t r_j,\quad\alpha:=\min_{1\le j\le t}k_j.
\]
Let $p>1$. Then there exists $n_0=n_0(F,p)$ such that, if $G$ is an $F$-free graph on $n\ge n_0$ vertices with $e_p(G)=\ex_p(n, F)$, then
\begin{equation}\label{eq:spiderforest1}
G
\begin{cases}
\in \mathcal L(n,t-1, \alpha-1), &\textup{\emph{if $F$ is a star forest, i.e., $r=k$;}}\\
=K_{t+k-r-1}\vee L', &\textup{\emph{if $\ell_{i,j}\in\{1, 3\}$ for all $i,j$, and}}\\
&\quad\quad\textup{\emph{$s\ge 1$, $m\ge 2$, $r\leq k-1$;}}\\
=L(n,t+k-r-1, 1), &\textup{\emph{if $\ell_{i,j}\in\{1, 3\}$ for all $i,j$, and $s=0$, $r\leq k-1$;}}\\
=L(n,\tau(F)-1,0), &\textup{\emph{if $F_{j_1}$ is either mixed parity or}}\\
&\quad\quad\textup{\emph{is an odd length path for some $j_1$;}}\\
=E(n,\tau(F)-1), &\textup{\emph{otherwise;}}
\end{cases}
\end{equation}
where $L'$ is the graph on $n-t-k+r+1$ vertices that consists of a matching of size $m-1$ and isolated vertices, and $\tau(F)$ is given by\\[-1.2ex]
\begin{equation}
\tau(F)=\sum_{j=1}^t\tau(F_j),\label{eq:spiderforest2}
\end{equation}
where for $1\le j\le t$,
\begin{equation}\label{eq:spiderforest3}
\tau(F_j)=
\left\{
\begin{array}{ll}
\sum_{i=1}^{k_j}\frac{\ell_{i,j}}{2}, & \textup{\emph{if $F_j$ is even-legged;}}\\[0.5ex]
\sum_{i=1}^{k_j}\big\lfloor\frac{\ell_{i,j}}{2}\big\rfloor+1, & \textup{\emph{if $F_j$ is odd-legged or mixed parity.}}
\end{array}
\right.
\end{equation}
Moreover, for the case when $F$ is a star forest, every graph of $\mathcal L(n,t-1,\alpha-1)$ is an extremal graph for $\ex_p(n,F)$.
\end{thm}

\begin{proof}
Let us refer the five cases of (\ref{eq:spiderforest1}) as (\ref{eq:spiderforest1}a) to (\ref{eq:spiderforest1}e). It is not hard to show that $\tau(F)$ is given by (\ref{eq:spiderforest2}) and (\ref{eq:spiderforest3}), and \\[-1.2ex]
\begin{equation}\label{eq:spiderforest4}
\ell(F)=
\left\{
\begin{array}{ll}
\alpha, & \textup{if $F$ is a star forest;}\\
1, & \textup{if $F_{j_1}$ is either mixed parity or}\\
& \quad\quad\textup{an odd length path, for some $j_1$;}\\
2, & \textup{otherwise.}
\end{array}
\right.
\end{equation}

We will consider the graph $K_{\tau(F)-1,n-\tau(F)+1}$, graphs of $\mathcal L(n,\tau(F)-1,\ell(F)-1)$, and the graph $E(n,\tau(F)-1)$. Let $U$ and $X$ be the sets of size $\tau(F)-1$ and $n-\tau(F)+1$ for these graphs. Again, we have $\ex(n, F) = O(n)$. Let $p > 1$, and choose $n_0 = n_0(F,p)$ according to Theorem \ref{thm:main} and Corollaries \ref{cor:H}, \ref{cor:L}. Let $n\ge n_0$.

Firstly, suppose that (\ref{eq:spiderforest1}a) holds. Then $\tau(F)=t\ge 2$ since $F$ is not a star, and $\ell(F)=\alpha$. We have every $H\in \mathcal L(n,t-1,\alpha-1)$ must be $F$-free, since if there is a copy of $F$ in $H$, then each of the $t$ stars in $F$ must use at least one vertex of $U$, which contradicts $|U|=\tau(F)-1=t-1$. Thus by Corollary \ref{cor:L}, we have $G\in \mathcal L(n,t-1,\alpha-1)$. Moreover, since $H$ was arbitrary, every graph of $\mathcal L(n,t-1,\alpha-1)$ is an extremal graph for $\ex_p(n,F)$. 

Next, suppose that (\ref{eq:spiderforest1}b) or (\ref{eq:spiderforest1}c) holds. From (\ref{eq:spiderforest2}) and  (\ref{eq:spiderforest3}), we have $\tau(F)=t+k-r\ge 2$. Also, since $F$ is not a star forest, we have $\ell(F)\le 2$ from (\ref{eq:spiderforest4}). By Theorem \ref{thm:main}, $G$ contains $K_{t+k-r-1,n-t-k+r+1}$ as a subgraph. Since $G$ is $F$-free, any two edges of $G$ in $X$ must be independent. If (\ref{eq:spiderforest1}b) holds, and there is a matching with $m\ge 2$ edges $e_1,\dots,e_m$ in $X$, then there is a copy of $F$. Indeed, assume that $F_1$ satisfies $k_1=m$, so that $r_1\le 1$. Let $v_1\in X$. Each leg of $F_1$ uses one of $e_1,\dots,e_m$, and the $k_1-r_1$ legs of length $3$ each uses one vertex of $U$. For $2\le j\le t$, let $v_j\in U$. The $k_j-r_j$ legs of length $3$ each uses one vertex of $U$, excluding $v_j$. All other vertices of $F$ are in $X$. The total number of vertices of $F$ in $U$ is $(t-1)+\sum_{j=1}^t(k_j-r_j)=t+k-r-1=|U|$. Hence, $G$ is a subgraph of $K_{t+k-r-1}\vee L'$.

Now, we show that $K_{t+k-r-1}\vee L'$ and $L(n,t+k-r-1,1)$ are $F$-free for (\ref{eq:spiderforest1}b) and (\ref{eq:spiderforest1}c) respectively. Suppose that there is a copy of $F$.  The set $X$ contains a matching of size $m-1$ for (\ref{eq:spiderforest1}b), and a maximum matching for (\ref{eq:spiderforest1}c). For $1\le j\le t$, let $m_j$ be the number of legs of $F_j$ that use one of the edges in $X$. If a centre $v_j\in U$, then $F_j$ has at least $k_j-r_j+1$ vertices in $U$. Suppose that a centre $v_j\in X$. For (\ref{eq:spiderforest1}b), we have $s\ge 1$. If $1\le j\le s$, then $r_j\le 1$, and $m_j\le m-1\le k_j-1$. For both (\ref{eq:spiderforest1}b) and (\ref{eq:spiderforest1}c), where $s=0$ for (\ref{eq:spiderforest1}c), if $s<j\le t$, then $r_j\ge 2$, so again $m_j\le k_j-(r_j-1)\le k_j-1$. In all cases, $F_j$ has at least $m_j-r_j+2(k_j-m_j)=2k_j-r_j-m_j\ge k_j-r_j+1$ vertices in $U$, whether or not exactly one leg of $F_j$ of length $1$ uses an edge in $X$ if $r_j\ge 1$. Thus, $F$ has at least $\sum_{j=1}^t(k_j-r_j+1)=t+k-r>|U|$ vertices in $U$, a contradiction. Hence, $G=K_{t+k-r-1}\vee L'$ if (\ref{eq:spiderforest1}b) holds, and by Corollary \ref{cor:L}, $G=L(n,t+k-r-1,1)$ if (\ref{eq:spiderforest1}c) holds.

Now, suppose (\ref{eq:spiderforest1}d) holds. We have $\ell(F)=1$ from (\ref{eq:spiderforest4}). Suppose that $L(n,\tau(F)-1,0)$ has a copy of $F$. Then, $F_j$ has at least $\sum_{i=1}^{k_j}\big\lfloor\frac{\ell_{i,j}}{2}\big\rfloor+1$ vertices in $U$ if $v_j\in U$, and at least $\sum_{i=1}^{k_j}\big\lceil\frac{\ell_{i,j}}{2}\big\rceil$ if $v_j\in X$. Thus, $F_j$ has at least $\sum_{i=1}^{k_j}\frac{\ell_{i,j}}{2}$ vertices in $U$ if $F_j$ is even-legged, and at least $\sum_{i=1}^{k_j}\big\lfloor\frac{\ell_{i,j}}{2}\big\rfloor+1$ otherwise. From (\ref{eq:spiderforest2}) and (\ref{eq:spiderforest3}), it follows that the number of vertices of $F$ in $U$ is at least $\tau(F)>|U|$, a contradiction. Hence, $L(n,\tau(F)-1,0)$ is $F$-free, and by Corollary \ref{cor:L}, we have $G=L(n,\tau(F)-1,0)$.

Finally, suppose (\ref{eq:spiderforest1}e) holds. Since (\ref{eq:spiderforest1}d) does not hold, every $F_j$ must either be even-legged, or odd-legged with $k_j\ge 2$. Moreover, since (\ref{eq:spiderforest1}a), (\ref{eq:spiderforest1}b) and (\ref{eq:spiderforest1}c) do not hold, we may assume that $F_1$ is either even-legged with $\ell_{1,1}\ge 4$ or $\ell_{1,1}=\ell_{2,1}=2$, or is odd-legged with $\ell_{1,1}\ge 5$. Now, by Theorem~\ref{thm:main}, $G$ contains $K_{\tau(F)-1, n-\tau(F)+1}$ as a subgraph, and from (\ref{eq:spiderforest4}), we have $\ell(F)=2$. It follows that $G$ cannot have two incident edges in $X$. Suppose that $G$ has two independent edges $e$ and $f$ in $X$. We show that there is a copy of $F$. If $F_1$ is even-legged and $\ell_{1,1}\ge 4$, let the centre $v_1$ of $F_1$ be in $X$, with the leg of length $\ell_{1,1}$ using $e,f$ and $\frac{\ell_{1,1}}{2}-1$ vertices of $U$. Otherwise, let $v_1\in U$. If  $F_1$ is even-legged, let $e,f$ be in the two legs with length $\ell_{1,1}=2$ and $\ell_{2,1}=2$. If $F_1$ is odd-legged, let the leg of length $\ell_{1,1}\ge 5$ use $e,f$, and $\big\lfloor\frac{\ell_{1,1}}{2}\big\rfloor-1$ other vertices of $U$, excluding $v_1$. In both cases, any other leg of $F_1$ with length $\ell_{i,1}$ uses $\big\lfloor\frac{\ell_{i,1}}{2}\big\rfloor$ vertices in $U$, excluding $v_1$ in the latter case. The number of vertices of $F_1$ in $U$ is $\sum_{i=1}^{k_1}\frac{\ell_{i,1}}{2}-1$ if $F_1$ is even-legged, and $\sum_{i=1}^{k_1}\big\lfloor\frac{\ell_{i,1}}{2}\big\rfloor$ if $F_1$ is odd-legged, and so this number is $\tau(F_1)-1$ from (\ref{eq:spiderforest3}). For every other $F_j$ where $j\ge 2$, we may have a copy of $F_j$ with $\tau(F_j)$ vertices in $U$. Hence by (\ref{eq:spiderforest2}), we may have a copy of $F$ with $\tau(F)-1=|U|$ vertices in $U$, and all remaining vertices in $X$.

It follows that $G$ is a subgraph of $E(n,\tau(F)-1)$. We show that $E(n,\tau(F)-1)$ is $F$-free. Suppose that there is a copy of $F$ in $E(n,\tau(F)-1)$, and recall that every $F_j$ is either even-legged, or odd-legged with $k_j\ge 2$. Let $e$ be the edge in the set $X$ of $E(n,\tau(F)-1)$, and assume that only $F_1$ may use $e$. We count the number of vertices of $F$ in $U$. Using (\ref{eq:spiderforest3}), if $F_1$ is even-legged, then $F_1$ has at least $\sum_{i=1}^{k_1}\frac{\ell_{i,1}}{2}=\tau(F_1)$ vertices in $U$, whether $v_1\in U$ or $v_1\in X$. If $F_1$ is odd-legged, then $F_1$ has at least $\sum_{i=1}^{k_1}\big\lfloor\frac{\ell_{i,1}}{2}\big\rfloor+1=\tau(F_1)$ vertices in $U$ if $v_1\in U$, and at least $\sum_{i=1}^{k_1}\big\lceil\frac{\ell_{i,1}}{2}\big\rceil-1=\sum_{i=1}^{k_1}\big\lfloor\frac{\ell_{i,1}}{2}\big\rfloor+k_1-1\ge \sum_{i=1}^{k_1}\big\lfloor\frac{\ell_{i,1}}{2}\big\rfloor+1=\tau(F_1)$ if $v_1\in X$. As in case (\ref{eq:spiderforest1}d), every other $F_j$ where $j\ge 2$ must use at least $\tau(F_j)$ vertices of $U$. Hence by (\ref{eq:spiderforest2}), $F$ uses at least $\tau(F)>|U|$ vertices of $U$, a contradiction. We conclude that $G=E(n,\tau(F)-1)$ if (\ref{eq:spiderforest1}e) holds.

This completes the proof of Theorem~\ref{thm:Spider}.
\end{proof}

\section*{Acknowledgements}
Ping Hu is partially supported by National Key Research and Development Program of China (No.~2021YFA1002100) and National Natural Science Foundation of China (No.~12471337). Henry Liu is supported by Guangdong and Hong Kong Universities ``1+1+1'' Joint Research Collaboration Project (No.~2025A0505000014).

\end{document}